\documentclass[10pt,twoside,a4paper]{article}

\usepackage[utf8]{inputenc}

\usepackage[T1]{fontenc}

\usepackage{amsmath}
\numberwithin{equation}{section}

\usepackage{amssymb}

\usepackage{microtype}

\usepackage{subcaption}

\usepackage[backend=bibtex,maxnames=10,backref=true,hyperref=true,giveninits=true,safeinputenc]{biblatex}
\DefineBibliographyStrings{english}{%
	backrefpage = {cited on page},
	backrefpages = {cited on pages},
}

\DeclareFieldFormat[report]{title}{``#1''}
\DeclareFieldFormat[book]{title}{``#1''}
\AtEveryBibitem{\clearfield{url}}
\AtEveryBibitem{\clearfield{note}}

\usepackage{graphicx}

\usepackage{tikz}
\usepackage{pgfplots}
\pgfplotsset{compat=newest}
\usetikzlibrary{external,shapes,arrows.meta}

\title{Inverse Problems of Trapped Objects}
\author{Peter Elbau$^1$\\{\footnotesize\href{mailto:peter.elbau@univie.ac.at}{peter.elbau@univie.ac.at}}
\and Monika Ritsch-Marte$^2$\\{\footnotesize\href{mailto:monika.ritsch-marte@i-med.ac.at}{monika.ritsch-marte@i-med.ac.at}}
\and Otmar Scherzer$^{1,3}$\\{\footnotesize\href{mailto:otmar.scherzer@univie.ac.at}{otmar.scherzer@univie.ac.at}}
\and Denise Schmutz$^1$\\{\footnotesize\href{mailto:denise.schmutz@univie.ac.at}{denise.schmutz@univie.ac.at}}}
\date{}

\usepackage[pdftex,colorlinks=true,linkcolor=blue,citecolor=green,urlcolor=blue,bookmarks=true,bookmarksnumbered=true]{hyperref}
\hypersetup
{
    pdfauthor={P. Elbau, M. Ritsch-Marte, O. Scherzer, D. Schmutz},
    pdfsubject={},
    pdftitle={Inverse Problems of Trapped Objects},
    pdfkeywords={Inverse Problems, Trapping}
}

\usepackage[hyperref,amsmath,thmmarks]{ntheorem}
\usepackage{aliascnt}
\usepackage{ntheorem}

\newtheorem{lemma}{Lemma}[section]

\newaliascnt{proposition}{lemma}
\newtheorem{proposition}[proposition]{Proposition}
\aliascntresetthe{proposition}

\newaliascnt{corollary}{lemma}

\aliascntresetthe{corollary}

\newaliascnt{theorem}{lemma}

\aliascntresetthe{theorem}

\theorembodyfont{\normalfont}
\newaliascnt{definition}{lemma}
\newtheorem{definition}[definition]{Definition}
\aliascntresetthe{definition}

\newaliascnt{assumption}{lemma}
\newtheorem{assumption}[assumption]{Assumption}
\aliascntresetthe{assumption}

\newtheorem*{example}{Example}

\theoremstyle{nonumberplain}
\theoremseparator{:}
\theoremheaderfont{\normalfont\itshape}

\newtheorem{remark}{Remark}

\theoremsymbol{\ensuremath{\square}}
\newtheorem{proof}{Proof}

\usepackage[a4paper,centering,bindingoffset=0cm,marginpar=2cm,margin=2.5cm]{geometry}

\usepackage[pagestyles]{titlesec}
\titleformat{\section}[block]{\large\sc\filcenter}{\thesection.}{0.5ex}{}[]
\titleformat{\subsection}[runin]{\bf}{\thesubsection.}{0.5ex}{}[.]

\newpagestyle{headers}%
{%
	\headrule
	\sethead%
		[\footnotesize\thepage]%
		[\footnotesize\sc P.~Elbau, M.~Ritsch-Marte, O.~Scherzer, D.~Schmutz]%
		[]%
		{}%
		{\footnotesize\sc Inverse Problems of Trapped Objects}%
		{\footnotesize\thepage}
	\setfoot{}{}{}
}
\usepackage[font=footnotesize,format=plain,labelfont=sc,textfont=sl,width=0.75\textwidth,labelsep=period]{caption}

\usepackage{dsfont}

\newcommand{\R}{\mathds{R}}
\newcommand{\C}{\mathds{C}}

\let\RE\Re
\let\Re=\undefined
\DeclareMathOperator{\Re}{\RE e}
\let\IM\Im
\let\Im=\undefined
\DeclareMathOperator{\Im}{\IM m}

\newcommand{\e}{\mathrm e}
\let\ii\i
\renewcommand{\i}{\mathrm i}
\renewcommand{\d}{\,\mathrm d}

\newcommand{\nspace}[1]{C_{c,+}(\R^{#1};\R)}
\newcommand{\Rspace}{C^4(\R;SO(3))}
\newcommand{\aspace}{C(\R;\R^3)}
\newcommand{\omspace}{C^3(\R;\R^3)}

\newcommand{\dk}[2]{{\bf D}^{#1} [ #2 ]}
\newcommand{\lsem}{[\![}
\newcommand{\lsemb}{\Big[\!\!\Big[}

\newcommand{\rsem}{]\!]}
\newcommand{\rsemb}{\Big]\!\!\Big]}

\begin{document}

\maketitle
\thispagestyle{empty}

\begin{center}
\parbox[t]{18em}{\footnotesize
\hspace*{-1ex}$^1$Faculty of Mathematics\\
University of Vienna\\
Oskar-Morgenstern-Platz 1\\
A-1090 Vienna, Austria}
\hfil
\parbox[t]{18em}{\footnotesize
	\hspace*{-1ex}$^2$Division for Biomedical Physics \\
	Medical University of Innsbruck \\
	Müllerstraße 44\\
	A-6020 Innsbruck, Austria}
\\[1ex]
\parbox[t]{18em}{\footnotesize
\hspace*{-1ex}$^3$Johann Radon Institute for Computational\\
\hspace*{1em}and Applied Mathematics (RICAM)\\
Altenbergerstraße 69\\
A-4040 Linz, Austria} 
\hfil
\parbox[t]{18em}{\mbox{}}
\end{center}

\begin{abstract}
 Optical and acoustical trapping has been established as a tool for holding and moving microscopic particles suspended in a 
 liquid in a contact-free and non-invasive manner. Opposed to standard microscopic imaging where the probe is fixated, this 
 technique allows imaging in a more natural environment. This paper provides a method for estimating the movement of a 
 transparent particle which is maneuvered by tweezers, assuming that the inner structure of the probe is not subject to local 
 movements. The mathematical formulation of the motion estimation shows some similarities to Cryo-EM single particle imaging, 
 where the recording orientations of the probe need to be estimated. 
\end{abstract}

\section{Introduction}\label{section_int}
In the past decades optical trapping has established itself as a versatile, and easy-to-implement tool for holding and moving 
microscopic ($\mu$m-sized) particles suspended in a liquid in a contact-free and non-invasive manner~\cite{JonMarVol15}.
The optical forces needed to balance the weight scale with the particle volume. Thus, for increasing particle size heating of the sample 
due to absorption of the trapping light leads to a problem, except for very transparent specimens. To hold heavier samples (for example 
cell-clusters, entire micro-organisms, or so-called ``organoids'', mini-organs grown in the petri-dish) other trapping modalities, 
such as acoustic forces, have to be added as auxiliary trapping forces~\cite{ThaSteMeiHilBer11}.

In this paper, we study the problem of estimation the movement of trapped particles, as a preprocessing step for 3D tomographic imaging. 
Conveniently, the required tomographic projection images from different view-points can be taken by rotating the trapped specimen 
directly by means of the acoustic or optical trapping forces. 
Here we will assume to be in a setting for which the light propagation can be sufficiently well modelled by straight rays which only undergo 
spatially varying attenuation, as in X-ray tomography \cite{Nat01}. In this case the optical image resembles a projection image. This is, for instance, 
fulfilled in low numerical aperture imaging of biological samples with sufficient amplitude contrast. 

In contrast to classical computerized tomography, however, the motion of the particle in the trap is irregular: Since the particle is not 
completely immobilized in the trap and since the locally acting forces are typically inhomogeneous during an entire revolution, it undergoes 
a complicated motion described by an affine transformation with time-dependent - and unknown - parameters, the momentary translation and 
angular velocity vectors. The focus of this work is to provide, as a first step of the reconstruction, a recovery of the motion of the 
particle without knowledge of the inner structure. As soon as the motion is determined, the attenuation projection data can be properly 
aligned to give the 
three-dimensional X-ray-transform, see, for example, \cite{Hel99}, (at least for those directions which were attained during the motion) 
and from this, a standard reconstruction formula could be applied to recover the inner structure of the object.

A main ingredient for our reconstruction of the motion is the common line method known for the determination of the relative orientations of different 
projection images as it is used, for example, in cryogenic electron microscopy~\cite{Hee87,Gon88a}. An infinitesimal version of this technique, 
allows us (after correcting for a potential translation by tracking the center of the images, as explained in \autoref{subsection_translation}) to 
find at every time step the local angular velocity of the motion, see \autoref{subsection_rotation}. Moreover, we provide a numerical test by presenting reconstructions of simulated data (in \autoref{sec:numerics}).

\section{Mathematical model}

We consider a bounded object in $\R^3$, characterised by an attenuation coefficient $u \in \nspace{3}$.
\begin{definition}
In this paper
\begin{equation*} 
 \nspace{3}=\{u\in C(\R^3;\R)\mid\mathrm{supp}(u)\ne\emptyset\;\text{is compact},\;u\ge0\}
\end{equation*}
denotes the space of \emph{admissible attenuation coefficients}. For $u \in \nspace{3}$ we define its \emph{center}
\begin{equation}\label{eq:center}
\mathcal C_3 := \frac{1}{\int_{\R^3}u(x)\d x} \int\limits_{\R^3}x u(x)\d x \in \R^3 \text{ with } x =\left(\begin{smallmatrix}
x_1\\x_2\\x_3	\end{smallmatrix}\right).
\end{equation} 
\end{definition}
In this paper we make the assumption that the object does not undergo a deformation during its motion, 
that is: 
\begin{assumption}
The attenuation coefficient $u$ is exposed to a \emph{rigid motion} consisting of a rotation $R\in C(\R;SO(3))$ and a translation $T\in C(\R;\R^3)$,
\begin{equation} \label{eq:affine}
A(t,x) = \mathcal C_3+R(t)(x-\mathcal C_3+T(t)).
\end{equation}
During such an induced motion, the object is illuminated from the $e_3$-direction with a \emph{uniform intensity}, see \autoref{fig:setup}. 
\end{assumption}

\begin{figure}
\begin{center}
		\includegraphics[scale=0.5]{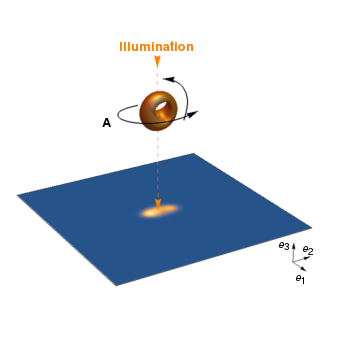}
\end{center}
	\caption{Illumination in the $e_3$-direction of the contour plot of the attenuation coefficient defined in \autoref{eq:defnumericsattcoeff} for $u(x)=42$ and the resulting attenuation projection image.}\label{fig:setup}
\end{figure}
Assuming in a rough approximation that the light moves along straight lines and only suffers from attenuation, we consider measurements in the form of \emph{attenuation projection mappings}.
\begin{definition}
	The \emph{attenuation mapping} $\mathcal J$ maps a rigid body motion (described by translation $t \mapsto T(t)$ and rotation $t \mapsto R(t)$)
	applied to an object of interest (described by the attenuation function $u:\R^3 \to \R$) onto the \emph{attenuation projection image}. That is the 
	\emph{attenuation mapping} of $u$ at time $t$ is given by 
	\begin{equation} \label{eq:J}
	(T,R) \mapsto \mathcal J[T,R](t,x_1,x_2)=\int_{-\infty}^\infty u\big(\mathcal C_3+R(t)(x-\mathcal C_3+T(t))\big)\d x_3.
	\end{equation}
\end{definition}

We will do the reconstruction of the motion in Fourier space, using the convention
\[\mathcal F_n[f](k) = (2\pi)^{-\frac n2}\int_{\R^n}f(x)\e^{-\i\left<k,x\right>}\d x \]
for the $n$-dimensional Fourier transform. 
For deriving the adequate formulas we need to introduce some notation first:
\begin{definition}
 The mapping 
 \begin{equation*}
 \begin{aligned}
  P: \R^3 &\to \R^2\\
  x &\mapsto \begin{pmatrix}x_1\\x_2\end{pmatrix}
  \end{aligned}
 \end{equation*}
 denotes the orthogonal projection onto the $x_1x_2-$plane. Moreover, the adjoint of $P$, $P^T : \R^2 \to \R^3$ is given by 
 $P^T k = (\begin{smallmatrix}k\\0\end{smallmatrix})$. 
\end{definition}
With this notation the attenuation mapping $\mathcal J$ satisfies \autoref{eqSetMapFourier}, below, which is an identity similar to the Fourier projection-slice theorem, see for example \cite{Bra56} and \cite[p.11]{Nat01}, that is used quite heavily for 
orientation reconstruction in Cryo-EM and $X$-Ray tomography.
\begin{lemma}\label{FSTforPM}
 Let $u\in \nspace{3}$ and $\mathcal J[R,T]$ be the attenuation mapping of a rigid body motion $(R,T)$, which the object of interest $u$
 gets exposed to. Then, the following identity holds:
 \begin{equation}\label{eqSetMapFourier}
  \mathcal F_2 [\mathcal J[T,R]]= \sqrt{2\pi}\,\mathcal F_3[u](R(t)P^Tk)\,\e^{\i\left<R(t)P^Tk,\mathcal C_3\right>}\e^{\i\left<k,P(T(t)-\mathcal C_3)\right>}.
 \end{equation}
 Here $\mathcal F_2$ denotes the two-dimensional Fourier transform with respect to the coordinates $(x_1,x_2)$. 
\end{lemma}
\begin{proof}
	We denote by 
	\begin{equation}\label{eq:hatj}
	\mathcal{\hat{J}} := \mathcal F_2 [\mathcal J[T,R]]
	\end{equation}
	the Fourier transform of the attenuation mapping of $u$ exposed to a rigid transformation. 
	By definition of $\mathcal{\hat{J}}$ and $\mathcal J$ it holds that
		\begin{align*}
		\mathcal{\hat{J}}&= \frac1{2\pi}\int_{\R^2}\mathcal J[T,R](t,x_1,x_2)\e^{-\i k_1x_1-\i k_2x_2}\d(x_1,x_2) \\
		&= \frac1{2\pi}\int_{\R^3}u\big(\mathcal C_3+R(t)(x-\mathcal C_3+T(t))\big)\e^{-\i\left<P^Tk,x\right>}\d x.
		\end{align*}
		Substituting $y=\mathcal C_3+R(t)(x-\mathcal C_3+T(t))$, we obtain the asserted relation
		\begin{align*}
		\mathcal{\hat{J}}&= \frac1{2\pi}\int_{\R^3}u(y)\e^{-\i\left<P^Tk,R(t)^T(y-\mathcal C_3)-(T(t)-\mathcal C_3)\right>}\d y \\
		&= \sqrt{2\pi}\,\mathcal F_3[u](R(t)P^Tk)\,\e^{\i\left<R(t)P^Tk,\mathcal C_3\right>}\e^{\i\left<k,P(T(t)-\mathcal C_3)\right>}.
		\end{align*}
\end{proof}
This means that the Fourier transform of the two-dimensional attenuation mapping $\mathcal J$ of a rigid motion $(T,R)$ of an object of interest $u$ coincides, up to a factor depending on the motion $A$, with the values of the 
three-dimensional Fourier transform of $u$ evaluated on the central slice spanned by the first two columns of $R(t)$. 

\section{Motion Estimation} \label{section_method}

In this section we present a method for reconstruction of the time dependent translation $t \mapsto T(t)$ and the rotation parameters of $t \mapsto R(t)$, 
respectively, from collected data of $\mathcal{\hat{J}}$, as defined in \autoref{eq:hatj}. 
For this purpose we proceed iteratively:
\begin{enumerate}
 \item In a first step, explained in \autoref{subsection_translation}, we exploit the centers of the attenuation mappings 
   in order to find the translations.
 \item We then define in \autoref{eqRecMapReduced} reduced attenuation mappings which do not depend on the translation anymore. 
 \item In \autoref{subsection_rotation} we use \autoref{CLPforRPM} in order to find the remaining rotational part.
\end{enumerate}
Instead of recovering the rotation matrix function $t \mapsto R(t)$ directly, we rather retrieve the corresponding angular velocity function 
$t \mapsto \omega(t)$, which is defined as follows:
\begin{definition}
	Let $R\in C^1(\R;SO(3))$. Then the \emph{angular velocity} $\omega\in C(\R;\R^3)$ corresponding to $R$ is defined via
	\begin{equation}\label{eqDefOmega}
	 R(t)^T R'(t)x = \omega(t)\times x\quad\text{for all}\quad x\in\R^3.
	\end{equation} 
	Moreover, we represent $\omega$ in cylindrical coordinates with cylindrical axis $e_3$ and denote by $\alpha\in C(\R;\R^+)$ the \emph{cylindrical radius} of $\omega$, and by 
	$v\in C(\R;\mathbb{S}^1)$ the \emph{cylindrical components} of $\omega$ and $\varphi\in C(\R;\R)$ the \emph{cylindrical angle} of $\omega$, 
	such that
	\begin{equation}\label{eqDefO_v}
	\omega(t)=\begin{pmatrix}
	\alpha(t) v(t)\\ \omega_3(t)
	\end{pmatrix}=\begin{pmatrix}
	\alpha(t) \cos(\varphi(t))\\	\alpha(t) \sin(\varphi(t))\\ \omega_3(t)
	\end{pmatrix}.
	\end{equation}
	Moreover, we set $v^\perp(t)=(-v_2(t),v_1(t))^T$.
\end{definition}

The general workflow is depicted in \autoref{fig:workflow}.
\begin{figure}[h]
		\begin{center}
\begin{tikzpicture}[auto,font=\footnotesize,node distance=2.75cm,
base/.style = {rectangle, rounded corners, draw=black,
	minimum width=2cm, minimum height=1cm,
	text centered, font=\sffamily},
calculate/.style ={base,fill=teal!50!lightgray},
recover/.style ={base,fill=green!25!lightgray},
result/.style ={base,fill=gray!20},
every node/.style={fill=white, font=\sffamily}, align=center]
\node (CalcCOM) [calculate] {Calculate\\Centers of AM};
\node (Reca) [recover, right of=CalcCOM] {Recover\\ Translation $T$};
\node (CalcRPM) [calculate, right of=Reca] {Calculate\\ RAM};
\node (Recvw3) [recover, right of=CalcRPM] {Recover\\ $v$ and $\omega_3$};
\node (Recal) [recover, right of=Recvw3] {Recover\\ $\alpha$};
\node (CalcR) [calculate, right of=Recal] {Calculate\\ Rotations $R$};
\draw[-Triangle] (CalcCOM) -- node[name=n1] {}  (Reca);
\draw[-Triangle] (Reca) -- (CalcRPM);
\draw[-Triangle] (CalcRPM) -- node[name=n2] {} (Recvw3);
\draw[-Triangle] (Recvw3) -- node[name=n3] {} (Recal);
\draw[-Triangle] (Recal) -- (CalcR);
\node (FST) [result,below of=n1] {\autoref{prop:COM}};
\node (CLE) [result,below of=Recvw3] {\autoref{CLPforRPM}};
\draw[->,>=stealth] (FST) -- (n1);
\draw[->,>=stealth] (CLE) -- node[below] {first\\ derivative} (n2);
\draw[->,>=stealth] (CLE) -- node[below] {third\\ derivative} (n3); 
\end{tikzpicture}
\end{center}
	\caption{Workflow. AM=Attenuation Mappings, RAM=Reduced Attenuation Mappings.}
	\label{fig:workflow}
\end{figure}
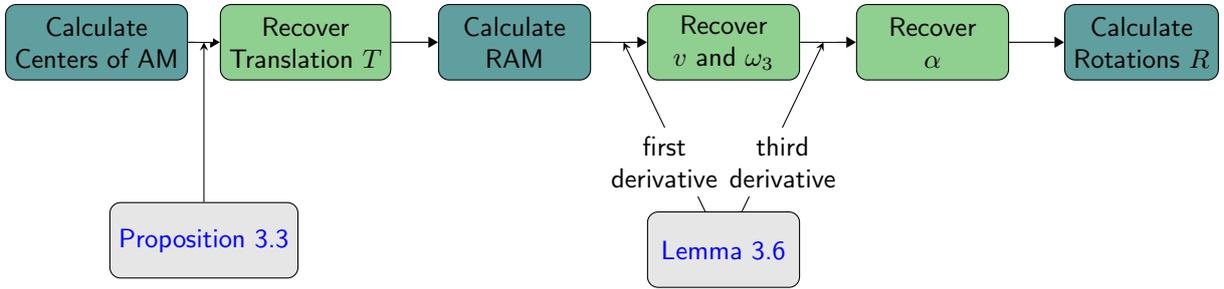

\subsection{Reconstruction of the Translation}\label{subsection_translation}
Since the attenuation operator integrates along the $e_3$ direction, we are facing a translational invariance of the 
attenuation projection data in this direction. This makes it impossible to uniquely reconstruct the third component of the 
translation $t \mapsto T(t)$ from the attenuation projection images:
\begin{lemma}\label{prop_translationinvariance}
	Let $u\in \nspace{3}$ and $\mathcal J$ be the attenuation mapping of a rigid motion of $u$. Then,
	\begin{equation*} 
	  \mathcal J[T,R] = \mathcal J[T+\rho e_3,R] \quad 
	\text{ for every } T\in C(\R;\R), \rho\in C(\R;\R) \text{ and } R\in C(\R;SO(3)).\end{equation*}
\end{lemma}
\begin{proof}
	Substituting $\tilde x_3=x_3+\rho(t)$, we find that
	\begin{align*}
	\mathcal J[T+\rho e_3,R](t,x_1,x_2) &= \int_{-\infty}^\infty u\big(\mathcal C_3+R(t)(x+\rho(t)e_3-\mathcal C_3+T(t))\big)\d x_3 \\
	&= \int_{-\infty}^\infty u\left(\mathcal C_3+R(t)\left(\left(\begin{smallmatrix}x_1\\x_2\\\tilde x_3\end{smallmatrix}\right)-\mathcal C_3+T(t)\right)\right)\d\tilde x_3 = \mathcal J[T,R](t,x_1,x_2).
	\end{align*}
\end{proof}

According to \autoref{prop_translationinvariance} it is impossible to reconstruct the third component of the translation $T$ of object of interest 
(that is the attenuation function), 
and thus we aim to recover only $P(T)$. 
It turns out that the orthogonal projection of the translated center of $u$ is equal to the center of the corresponding attenuation
mapping $\mathcal C_2$, which can be directly computed from the projection data. 
\begin{proposition}\label{prop:COM}
 Let $u\in \nspace{3}$ and $\mathcal J$ be the attenuation mapping of a rigid motion an object of interest $u$ is exposed to. 
 Moreover, let $\mathcal{C}_3$ be the center of $u$ defined in \autoref{eq:center} and $\mathcal{C}_2$ the \emph{two-dimensional 
 center} of the attenuation mapping $\mathcal{J}$,
 \begin{equation} \label{eq:c2}
  \mathcal{C}_2:=\frac{1}{\int_{\R^2} \mathcal J[T,R](t,x)\d x} 
  \int\limits_{\R^2}\left(\begin{smallmatrix} x_1\\x_2 \end{smallmatrix}\right)\mathcal J[T,R](t,x)\d x.
 \end{equation}	
 Then,
 \begin{equation}\label{eqRecCenter}
  P(\mathcal C_3-T(t)) = \mathcal C_2\quad\text{for every}\quad T\in \aspace, R\in C(\R;SO(3)),t\in\R.
 \end{equation}
\end{proposition}

\begin{proof}
	To calculate the center of $\mathcal J[T,R](t)$, we first remark that by substituting $y=\mathcal C_3+R(t)(x-\mathcal C_3+T(t))$
	it follows that
	\begin{equation}\label{eqRecTotalAtt}
	\int_{\R^2}\mathcal J[T,R](t,x_1,x_2)\d(x_1,x_2) = \int_{\R^3}u\big(\mathcal C_3+R(t)(x-\mathcal C_3+T(t))\big)\d x = \int_{\R^3}u(y)\d y.
	\end{equation}
	
	Moreover, we find
	\begin{equation*} \int_{\R^2}\begin{pmatrix}x_1\\x_2\end{pmatrix}\mathcal J[T,R](t,x_1,x_2)\d(x_1,x_2) = \int_{\R^3}Px\,u\big(\mathcal C_3+R(t)(x-\mathcal C_3+T(t))\big)\d x. \end{equation*}
	Substituting again $y=\mathcal C_3+R(t)(x-\mathcal C_3+T(t))$, we get that
	\begin{align*}
	\int_{\R^2}\begin{pmatrix}x_1\\x_2\end{pmatrix}&\mathcal J[T,R](t,x_1,x_2)\d(x_1,x_2) = \int_{\R^3}P\big(R(t)^T(y-\mathcal C_3)+\mathcal C_3-T(t)\big)\,u(y)\d y \\
	&= PR(t)^T\left(\int_{\R^3}yu(y)\d y-\mathcal C_3\int_{\R^3}u(y)\d y\right)+P(\mathcal C_3-T(t))\int_{\R^3}u(y)\d y
	\end{align*}
	Since the first term vanishes according to the definition of $\mathcal C_3$, we get with \autoref{eqRecTotalAtt} the identity
	\begin{equation*} \int_{\R^2}\begin{pmatrix}x_1\\x_2\end{pmatrix}\mathcal J[T,R](t,x_1,x_2)\d(x_1,x_2) = P(\mathcal C_3-T(t))\int_{\R^2}\mathcal J[T,R](t,x_1,x_2)\d(x_1,x_2), \end{equation*}
	which is just \autoref{eqRecCenter}.
\end{proof}
\begin{remark}
 From \autoref{eqRecCenter} we can reconstruct the translation part of $T$ in the $x_1x_2$-plane. According to \autoref{prop_translationinvariance} this 
 is the most we can hope for.
\end{remark}

\subsection{Reconstruction of the Rotation}\label{subsection_rotation}
In a second step we aim at recovering the cylindrical parameters $t\mapsto v(t)$ ($t \mapsto \varphi(t)$, respectively), $t\mapsto \alpha(t)$ and $t\mapsto\omega_3(t)$ corresponding to the rotation function $t \mapsto R(t)$ from the reduced attenuation mapping $\tilde{\mathcal J}$, 
which is the Fourier transform of $\mathcal{J}$ after translation correction with \autoref{eqRecCenter}:
\begin{definition} 
 We define the \emph{reduced attenuation map} corresponding to $u$ as
 \begin{equation}\label{eqRecMapReduced}
  \begin{aligned}
   \tilde{\mathcal J}: \R \times \R^2 &\to \R,\\
   (t,k) &\mapsto \mathcal F_2 [\mathcal J[T,R]](t,k)\,\e^{\i\left<k,\mathcal C_2\right>}.
  \end{aligned}
 \end{equation}
\end{definition}

\begin{lemma} 
 The reduced attenuation mapping is independent of the translation $t \mapsto T(t)$ and only dependent on the rotation $t \mapsto R(t)$.
\end{lemma}
\begin{proof}
 Multiplying \autoref{eqSetMapFourier} with $\e^{\i\left<k,\mathcal C_2\right>}$ and taking into account the relation between $\mathcal{C}_2$ and 
 $\mathcal{C}_3$ from \autoref{eqRecCenter} we get 
 \begin{equation*}
  \tilde{\mathcal J}(t,k) = \sqrt{2\pi}\,\mathcal F_3[u](R(t)P^Tk)\,\e^{\i\left<R(t)P^Tk,\mathcal C_3\right>},
 \end{equation*}
 where the right hand side is independent of the translation $T$.
\end{proof}
The reduced attenuation mapping possesses the following symmetry property.
\begin{lemma}\label{CLPforRPM}
 Let $u\in \nspace{3}$ and let $\tilde{\mathcal J}$ be the reduced attenuation mapping. 
 Then, for arbitrary $R\in C(\R;SO(3))$ the following identity holds
 \begin{equation}\label{eqRecSymmetry}
  \tilde{\mathcal J}\left(s,\frac\lambda{t-s}P(e_3\times(R(s)^TR(t)e_3))\right) = 
  \tilde{\mathcal J}\left(t,\frac\lambda{s-t}P(e_3\times(R(t)^TR(s)e_3))\right)
 \end{equation}
 for all $\lambda\in\R$ and $s,t\in\R$ with $s\ne t$.
\end{lemma}

\begin{proof}
 We remark that we have for arbitrary $s,t\in\R$ the relation
 \begin{equation*} 
  (R(t)e_3)\times(R(s)e_3) = R(t)(e_3\times(R(t)^TR(s)e_3)) = R(t)P^TP(e_3\times(R(t)^TR(s)e_3)). 
 \end{equation*}
 Choosing 
 $$k=\frac\lambda{s-t}P(e_3\times(R(t)^TR(s)e_3)) \quad \text{ for } s\ne t$$ with some arbitrary $\lambda\in\R$, 
 we find from \autoref{eqSetMapFourier} that
	\begin{equation} \label{eq:h}
	\tilde{\mathcal J}\left(t,\frac\lambda{s-t}P(e_3\times(R(t)^TR(s)e_3))\right)
	= \sqrt{2\pi}\mathcal F_3[u]\left(\frac\lambda{s-t}(R(t)e_3)\times(R(s)e_3)\right)\e^{\i\frac\lambda{s-t}\left<(R(t)e_3)\times(R(s)e_3),\mathcal C_3\right>}.
	\end{equation}
	Since by the definition of the cross product 
	\begin{equation*} 
	 (R(t)e_3)\times(R(s)e_3) = -(R(s)e_3)\times(R(t)e_3), 
	\end{equation*}
	the right hand side of \autoref{eq:h} is invariant with respect to interchanging $s$ and $t$, and we get \autoref{eqRecSymmetry}.
\end{proof}

The above result is the starting point for the recovery of the rotation matrix function $t \mapsto R(t)$ from measurement of the attenuation 
mapping $\mathcal{J}$ defined in \autoref{eq:J}. The basic idea is to differentiate \autoref{eqRecSymmetry} with respect to time up 
to order three. We see below that from the first order derivative of the equation it is possible to find the cylindrical component $v$ 
and the height $\omega_3$ of the angular velocity $\omega$ (see \autoref{propD1}). 

\subsubsection*{Reconstruction of the cylindrical component $v$ and the height $\omega_3$}
We show below that by differentiation of \autoref{eqRecSymmetry} with respect to $t$ we can recover the cylindrical components $v$ 
and $\omega_3$ of the angular velocity $\omega$ as defined in \autoref{eqDefO_v} and \autoref{eqDefOmega}, respectively.
In this section we use the tensor derivative notation as summarized in \autoref{sec:app}.

First we derive the derivative of \autoref{eqRecSymmetry}:
\begin{proposition} \label{propD1}
 Let $u\in \nspace{3}$ and $\tilde{\mathcal J}$ as defined in \autoref{eqRecMapReduced}. 
 Moreover, let $R\in C^2(\R;SO(3))$ and $\omega\in C^1(\R;\R^3)$ the associated angular velocity as defined in \autoref{eqDefOmega}. 
 Then, for all $t\in\R$ satisfying $\alpha(t) \ne 0$ and all $\mu \in \R$ the following relation holds:
 \begin{equation}\label{eqRecSymmetryFirstOrder}
  \partial_t\tilde{\mathcal J}(t,\mu v(t))=\omega_3(t)\dk{1}{\tilde{\mathcal J}}(t,\mu v(t))\lsem\mu v^\perp(t)\rsem.
 \end{equation}
\end{proposition}

\begin{proof}
 We insert in \autoref{eqRecSymmetry} the first order Taylor polynomials
 \begin{align*}
  \frac1{t-s}P(e_3\times(R(s)^TR(t)e_3)) &= a_0(s)+a_1(s)(t-s)+o(t-s)\quad\text{and} \\
  \frac1{s-t}P(e_3\times(R(t)^TR(s)e_3)) &= b_0(s)+b_1(s)(t-s)+o(t-s)
 \end{align*}
 with the coefficients $a_j$ and $b_j$, $j=0,1$, calculated in \autoref{thTayPolynomial}. 
 Then, by differentiating \autoref{eqRecSymmetry} with respect to $t$ at the position $t=s$, we 
 find for every $\lambda\in\R$ the relation
 \begin{equation*} 
  \dk{1}{\tilde{\mathcal J}}(s,\lambda a_0(s))\lsem\lambda a_1(s)\rsem = 
  \partial_t\tilde{\mathcal J}(s,\lambda b_0(s))+ \dk{1}{\tilde{\mathcal J}}(s,\lambda b_0(s))\lsem\lambda b_1(s)\rsem. 
 \end{equation*}
 Inserting the expressions from \autoref{eqTayCoefficients} for $a_0=\alpha v = b_0$ and $2a_1 = \omega_3 \alpha v^\bot +(\alpha v)'$, 
 $2b_1 = -\omega_3 \alpha v^\bot +(\alpha v)'$, we end up with the identity
 \begin{equation*} 
 \partial_t\tilde{\mathcal J}(s,\lambda\alpha(s)v(s)) = 
 \dk{1}{\tilde{\mathcal J}}(s,\lambda\alpha(s)v(s))\lsem\lambda\omega_3(s)\alpha(s)v^\perp(s)\rsem \quad \text{ for all }\lambda,s\in\R.
 \end{equation*}
 If we choose for an $s\in\R$ with $\alpha(s)\ne0$ the parameter $\lambda=\frac\mu{\alpha(s)}$, this gives us \autoref{eqRecSymmetryFirstOrder}.
\end{proof}

It is possible to recover for every time $t\in\R$ the cylindrical components $v(t)$ and the cylindrical height $\omega_3(t)$ using \autoref{propD1}. Since \autoref{eqRecSymmetryFirstOrder} holds for every $\mu \in \R$ it is sufficient look for a vector $v(t)\in \mathbb{S}^1$ such that the function
\begin{equation}\label{eq:constantfunctionomega3}
 \mu\mapsto\frac{\partial_t\tilde{\mathcal J}(t,\mu v(t))}{\dk{1}{\tilde{\mathcal J}}(t,\mu v(t))\lsem \mu v^\perp(t)\rsem}
\end{equation}
is constant for every $\mu\in\R$. This is in principle possible as we show in \autoref{sec:numerics}. The value of this constant function will then be $\omega_3(t)$. Of course we have to assume that there does not exist a vector $v\in \mathbb{S}^1$ such that $\partial_t[\tilde{\mathcal J}](t,\mu v)=0$ and 
$\dk{1}{\tilde{\mathcal J}}(t,\mu v)\lsem\mu v^\perp\rsem=0$ for all $\mu \in \R$. The question remains open under what conditions it is possible to determine $v$, and consequently $\omega_3$, uniquely? It turns out that we have to face a similar non-uniqueness issue as in Cryo-EM, where the reconstruction of the rotations is only possible up to an orthogonal transformation \cite{Lam08}. This follows from the fact that the object of interest, in our case described by the attenuation coefficient $u$, is unknown as well and a reflection of $u$ in the $x_1x_2$-plane through the origin leads to the same attenuation projection data. In fact, this ambiguity can directly be observed from \autoref{eqRecSymmetryFirstOrder}, which is obviously solved by $v(t)$ and $\check v(t)=-v(t)$. Denoting by $\omega(t)$ and $\check{\omega}(t)$ the corresponding angular velocities and by $R(t)$ and $\check R(t)$ the solutions to \autoref{eqDefOmega} with initial conditions $R(0)=\check R(0)=\mathds1$, we get the following connection
\[\check R(t)=\Sigma R(t) \Sigma,\]
where $\Sigma=\mathrm{diag}(1,1,-1)$.
When the object is moved with respect to $R$ and the reflected object with respect to $\check R$, we get exactly the same 
attenuation projection data (see \autoref{fig:reflection}).
\begin{figure}[h]
\begin{center}
		\begin{subfigure}[t]{0.4\textwidth}
		\centering
		\includegraphics[scale=0.4]{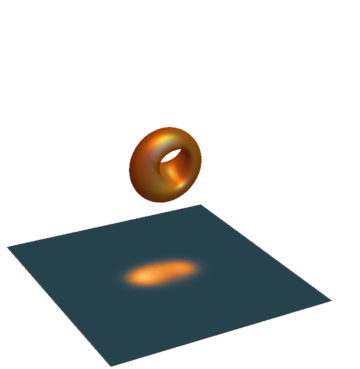}
		\caption{}
		\label{subfig:A}
	\end{subfigure}
	\begin{subfigure}[t]{0.4\textwidth}
		\centering
		\includegraphics[scale=0.4]{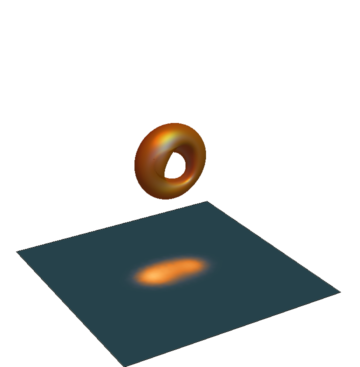}
		\caption{}
	\end{subfigure}
\vspace*{0.5cm}
	\begin{subfigure}[t]{0.4\textwidth}
		\centering
		\includegraphics[scale=0.4]{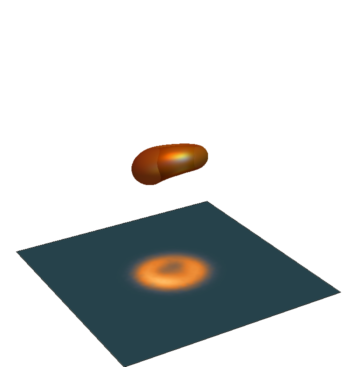}
		\caption{}
	\end{subfigure}
	\begin{subfigure}[t]{0.4\textwidth}
		\centering
		\includegraphics[scale=0.4]{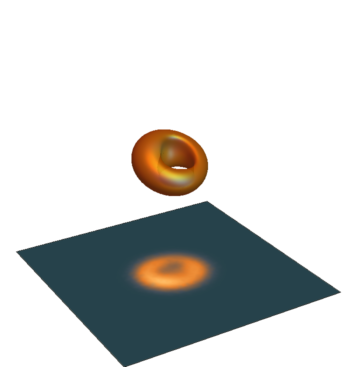}
		\caption{}
	\end{subfigure}
\end{center}
	\caption{(A) and (C): Object rotated with respect to $R$ at two different time steps.\\(B) and (D): Reflected object rotated with respect to $\check{R}$ at the same time steps.\\ The attenuation projection images are the same.}
	\label{fig:reflection}
\end{figure}

\subsubsection*{Reconstruction of the cylindrical radius $\alpha$}
So far we have explained how to recover the cylindrical parameters $t\mapsto v(t)$ (or equivalently $t \mapsto \varphi(t)$) and $t \mapsto \omega_3(t)$. 
For the full reconstruction of the motion will still need to estimate the cylindrical radius $t \mapsto \alpha(t)$, which can be done by using 
\autoref{propD3} below. Note that we go directly to the third derivative of \autoref{eqRecSymmetry}, as the second derivative provides no new information.

\begin{proposition}
	\label{propD3}
	Let $u\in \nspace{3}$, $\tilde{\mathcal J}$ be the reduced attenuation mapping of a rigid motion of $u$ in Fourier space. Let further $R\in\Rspace$, $t\in\R$ and $\omega\in\omspace$ be the angular velocity corresponding to $R$ and let $\sigma(t)=\varphi'(t)$.
		Then, for all $t\in\R$ such that $\alpha(t)\ne0$ and $\sigma(t)\ne-\omega_3(t)$, we have
	\begin{equation}
	\label{eq:RecSymmetryThirdOrder}
	A_0(\mu)+A_{02}(\mu)\alpha(t)^2+A_{1}(\mu)\mu\frac{\alpha'(t)}{\alpha(t)}=0\quad\text{for all}\quad \mu\in\R,
	\end{equation}
	where
	\begin{align*}
	A_0(\mu) &=\frac14\mu(\omega_3+\sigma)\Big[\mu^2\omega_3(\omega_3-\sigma)\dk{3}{\tilde{\mathcal J}}(s,\mu v)\lsem v^\perp,v^\perp,v^\perp\rsem \nonumber \\
	&\qquad +2\mu \dk{2}{\tilde{\mathcal J}}(s,\mu v)\lsem v^\perp,\omega_3\sigma v-\omega_3'v^\perp\rsem+2\dk{1}{\tilde{\mathcal J}}(s,\mu v)\lsem\omega_3^2v^\perp+\omega_3'v\rsem \\
	&\qquad -\mu(3\omega_3-\sigma)\partial_t\dk{2}{\tilde{\mathcal J}}(s,\mu v)\lsem v^\perp,v^\perp\rsem+2\partial_{tt}\dk{1}{\tilde{\mathcal J}}(s,\mu v)\lsem v^\perp\rsem\Big], \nonumber \\
	A_{02}(\mu) &= \frac12\mu(\omega_3+\sigma)\dk{1}{\tilde{\mathcal J}}(s,\mu v)\lsem v^\perp\rsem,\\
	A_1(\mu) &= \frac12(\omega_3+\sigma)\Big[\mu\omega_3\dk{2}{\tilde{\mathcal J}}(s,\mu v)\lsem v^\perp,v^\perp\rsem
	-\omega_3\dk{1}{\tilde{\mathcal J}}(s,\mu v)\lsem v\rsem-\partial_t\dk{1}{\tilde{\mathcal J}}(s,\mu v)\lsem v^\perp\rsem\Big].	
	\end{align*}
	Note that we omitted the dependence on time of the coefficients for the sake of readability.
\end{proposition}
\begin{proof}
 The proof is provided in the \autoref{sec:app}.
\end{proof}

It is possible to recover for every time $t\in\R$ the parameter $\alpha(t)$ using \autoref{propD3}. Since \autoref{eq:RecSymmetryThirdOrder} holds for every $\mu\in\R$, we can consider it as an overdetermined linear system for $\alpha^2(t)$ and $\frac{\alpha'(t)}{\alpha(t)}$, see \autoref{sec:numerics}.

\begin{example}[The case $\omega_3=-\sigma$]
Apparently, the coefficients $A_0,A_{02}$ and $A_1$ in \autoref{propD3} vanish, if $\omega_3=-\sigma$, in which case 
\autoref{eq:RecSymmetryThirdOrder} becomes trivial. Let us consider this case in more detail. 
For given cylindrical components $t \mapsto \omega(t)$, $t \mapsto v(t)$ ($t \mapsto \varphi(t)$, respectively), 
defined in \autoref{eqDefO_v} we get in this special situation
$\sigma(t)=\varphi'(t)=-\omega_3(t)$. Now, we calculate the rotation $R$, defined in \autoref{eqDefOmega} for given $t \mapsto \omega(t)$.
To simplify the calculations, we assume that $\alpha'=0$ and $\omega_3'=0$. Solving \autoref{eqDefOmega} with the initial condition $R(0)=\mathds1$ then gives
\begin{equation*} R(t) = \begin{pmatrix}\cos(\omega_3 t)&-\sin(\omega_3 t)&0\\ -\sin(\omega_3t)\cos(\alpha t)&-\cos(\omega_3t)\cos(\alpha t)&\sin(\alpha t)\\\sin(\omega_3 t)\sin(\alpha t)&\cos(\omega_3t)\sin(\alpha t)&\cos(\alpha t)\end{pmatrix}. \end{equation*}
We calculate
\begin{equation*} P\left(e_3\times(R(s)^TR(t)e_3)\right)=\sin(\alpha(t-s))\begin{pmatrix}\cos(\omega_3s)\\-\sin(\omega_3s)\end{pmatrix}\end{equation*}
and circle back to \autoref{eqRecSymmetry}, which then reads
\begin{equation*}
\tilde{\mathcal J}\left(s,\frac\lambda{t-s}\sin(\alpha(t-s))\begin{pmatrix}\cos(\omega_3s)\\-\sin(\omega_3s)\end{pmatrix}\right) = \tilde{\mathcal J}\left(t,\frac\lambda{s-t}\sin(\alpha(t-s))\begin{pmatrix}-\cos(\omega_3t)\\\sin(\omega_3t)\end{pmatrix}\right)
\end{equation*}
and holds for all $\lambda\in\R$. Since the term $\sin(\alpha(t-s))$ can be absorbed into the variable $\lambda$, this equation contains no information about $\alpha$. Therefore, it is not possible to recover this special motion completely with our technique.
\end{example}

\section{Numerics}\label{sec:numerics}
We consider for the points $P_1=(1,\frac12,-1)$, $P_2=(-\frac12,1,1)$, $P_3=(0,-1,\frac12)$ and the diagonal matrix $D=\mathrm{diag}(\sqrt2,1,1)$ as an example the attenuation coefficient
\begin{equation}\label{eq:defnumericsattcoeff}
u(x) = \prod_{i=1}^3|x-P_i|^2\e^{-\frac14|Dx|^2}
\end{equation}
and define the motion $T$ by the choice
\[ a(t) = \begin{pmatrix}\cos(6t)\cos(12t)\\\cos(6t)\sin(12t)\\\sin(t)\end{pmatrix}\text{ and }\omega(t)=\begin{pmatrix}\alpha(t)\cos(\varphi(t))\\\alpha(t)\sin(\varphi(t))\\\omega_3(t)\end{pmatrix} \]
with
\[ \alpha(t)=1+10t^2,\;\varphi(t) = \pi t+\tfrac\pi3,\text{ and }\omega_3(t) = \tfrac12+\sqrt{\tfrac12+5t}. \]

We discretise the transformed function by evaluating it at the points
\[ (j_1,j_2,j_3)\delta_x,\;j\in\{-512,\ldots,511\}^2\times\{-256,\ldots,255\}\text{ with }\delta_x=0.05 \]
and calculate the attenuation projection along the third direction at every time step
\[ \ell\,\delta_t,\;\ell\in\{0,\ldots,999\}\text{ for }\delta_t=0.0005. \]
Some attenuation projection images are depicted in \autoref{fig:projections}.

\begin{figure}
	\includegraphics[scale=0.28]{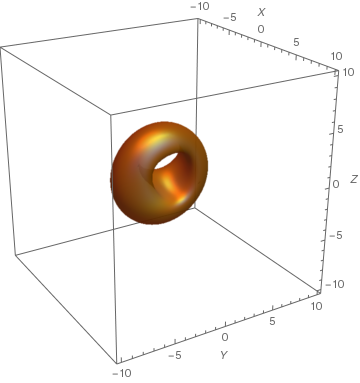}
	\includegraphics[scale=0.28]{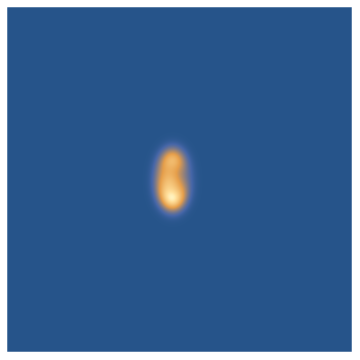}
	\includegraphics[scale=0.28]{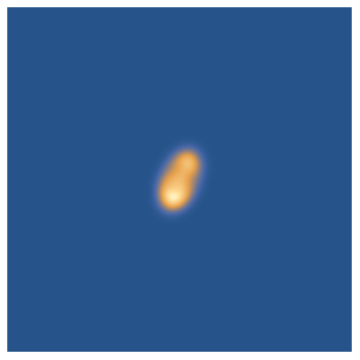}
	\includegraphics[scale=0.28]{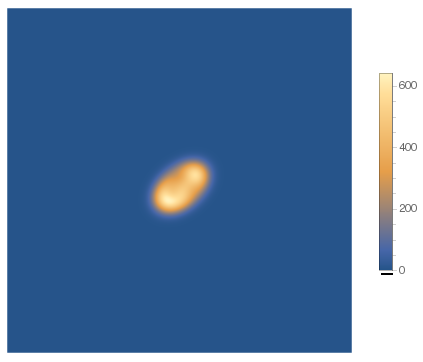}
	\caption{Contour plot of the attenuation coefficient defined in \autoref{eq:defnumericsattcoeff} for $u(x)=42$ and the resulting attenuation projection images at time steps $0,0.25$ and $0.425$.}
	\label{fig:projections}
\end{figure}

For the reconstruction, we proceed for each time step as follows:
\begin{enumerate}
\item
We calculate the center of the projections and read off the first two components of the displacement $a(t)$ via \autoref{eqRecCenter}, see \autoref{fgRecCenter}.
\begin{figure}
\begin{center}
\includegraphics{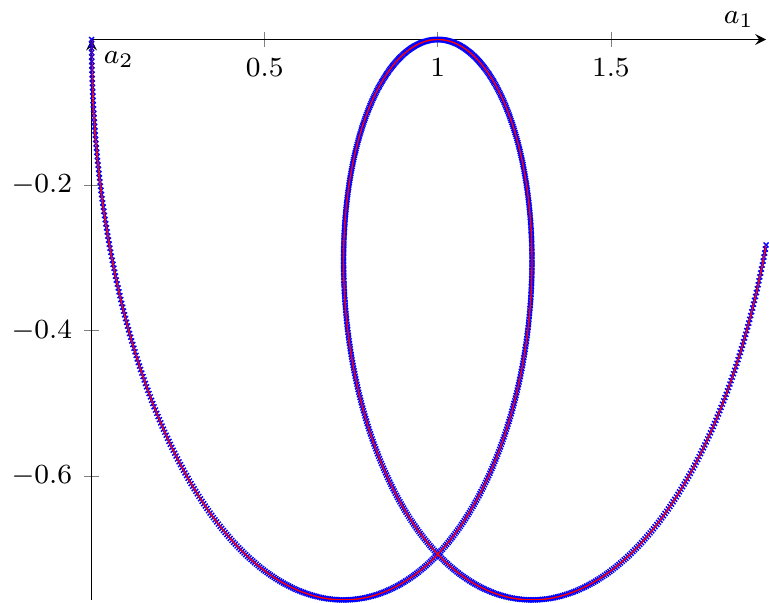}
\end{center}
\caption{Reconstruction of the function $a$ (the blue crosses are the reconstructed values, the red curve is the exact function).}\label{fgRecCenter}
\end{figure}
\item
To reconstruct the third component $\omega_3(t)$ as function of the yet unknown angular value, we interpret 
\autoref{eqRecSymmetryFirstOrder} as a least square minimisation problem for the function
\[ 
\sum_{j=-512}^{511} 
 \left|\partial_t\tilde J\left(t,j\delta_x\begin{pmatrix}\cos(\varphi(t))\\\sin(\varphi(t))\end{pmatrix}\right)-\omega_3(t)
  \dk{1}{\tilde J}\left(t,j\delta_x\begin{pmatrix}\cos(\varphi(t))\\\sin(\varphi(t))\end{pmatrix}\right)
  \lsemb 
  j\delta_x\begin{pmatrix}-\sin(\varphi(t))\\ \cos(\varphi(t))\end{pmatrix} \rsemb \right|^2, \]
where we remark that $\tilde{\mathcal J}$ and its derivatives can be nicely interpolated, with the help of the Whittaker--Shannon interpolation formula, for example, as they are Fourier transforms of a function which is very rapidly decreasing. We call the minimisation point $\tilde\omega_3(\varphi(t))$.
\item
To obtain the angular function $\varphi(t)$, we minimise the function
\begin{multline*}
 \Phi:\phi\mapsto\max_{j\in\{-512,\ldots,511\}} \Bigg\{ \left( \left| 
   \tilde\omega_3(\phi)\dk{1}{\tilde J} \left( t, j \delta_x \begin{pmatrix}\cos(\phi)\\\sin(\phi)\end{pmatrix}\right) 
   \lsemb j\delta_x\begin{pmatrix}-\sin(\phi)\\\cos(\phi)\end{pmatrix} \rsemb\right|^2+\varepsilon\right)^{-1} \\
\times
\left|\partial_t\tilde J\left(t,j\delta_x\begin{pmatrix}\cos(\phi)\\\sin(\phi)\end{pmatrix}\right)-\tilde\omega_3(\phi)
\dk{1}{\tilde J}\left(t,j\delta_x\begin{pmatrix}\cos(\phi)\\\sin(\phi)\end{pmatrix}\right)\lsemb j\delta_x\begin{pmatrix}-\sin(\phi)\\\cos(\phi)\end{pmatrix}\rsemb\right|^2\Bigg\},
\end{multline*}
on $[0,\pi)$ with some tiny $\varepsilon>0$. Since the function is quite complicated and the minimisation problem is only one-dimensional, see \autoref{fgResidual}, we chose to minimise by brute force, that is, we simply evaluate it for sufficiently many values $\phi$ and pick the one with smallest value. The minimiser gives us $\varphi(t)$ and thus $\omega_3(t)=\tilde\omega_3(\varphi(t))$, see \autoref{fgRecAngle}.
\begin{figure}
\begin{center}
\includegraphics{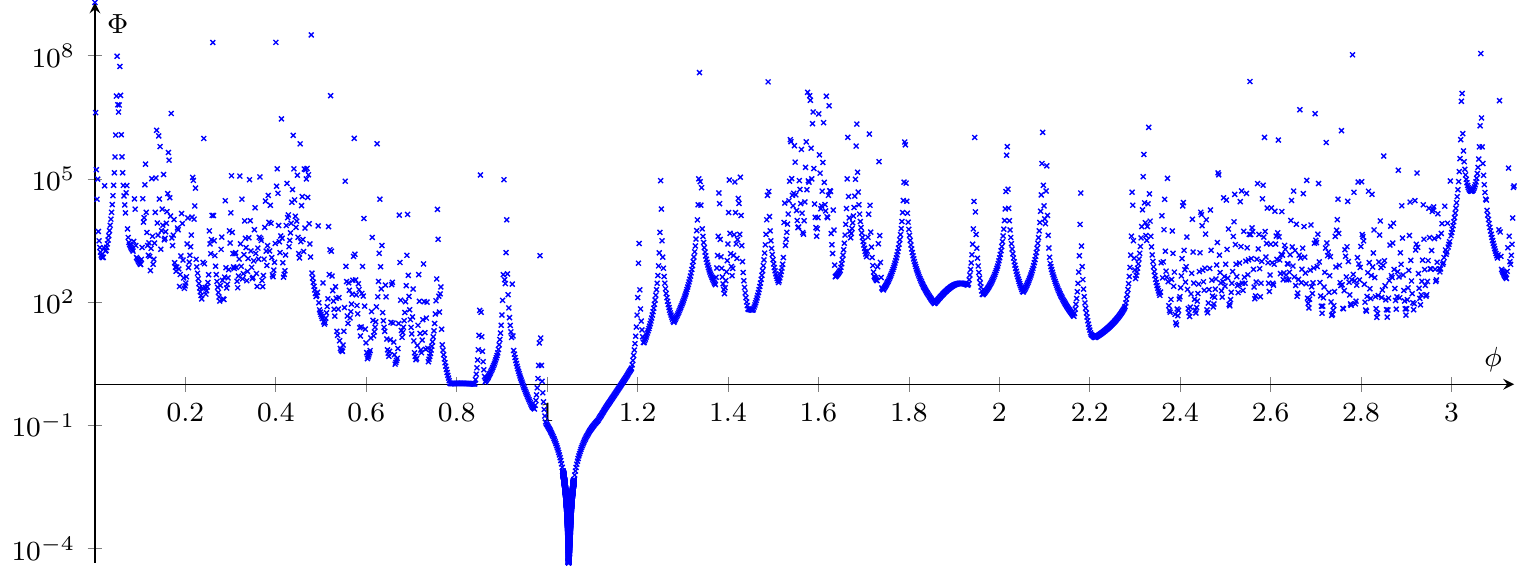}
\end{center}
\caption{Graph of the function $\Phi$ at the time step $\delta_t$.}\label{fgResidual}
\end{figure}

\begin{figure}
\includegraphics{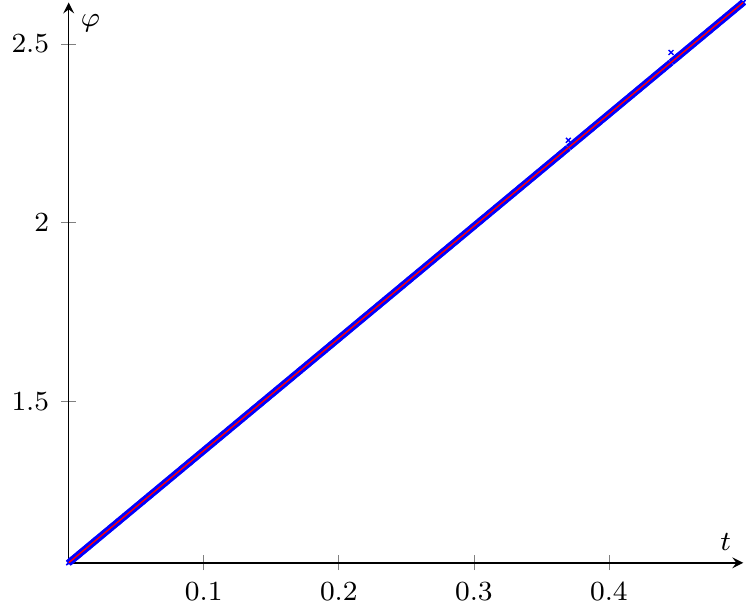} \hfill \includegraphics{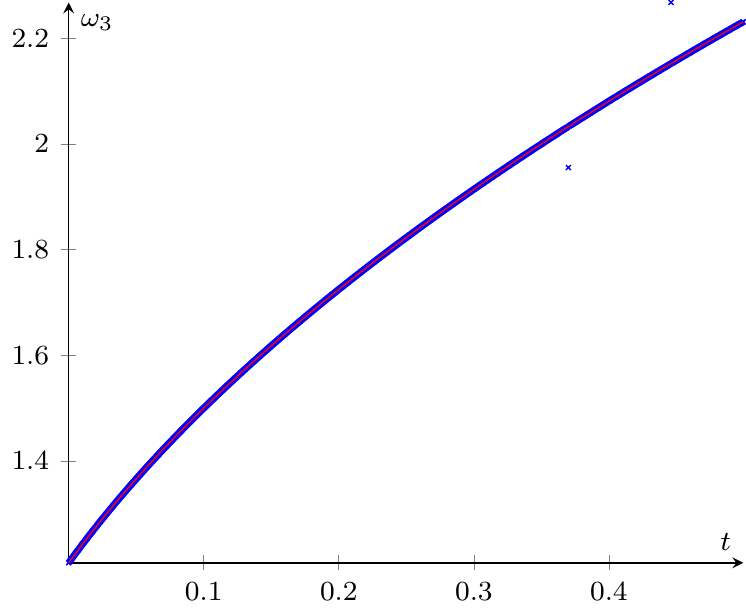}
\caption{Reconstruction of $\varphi$ on the left and $\omega_3$ on the right (the blue crosses are the reconstructed values, the red curve is the exact function).}\label{fgRecAngle}
\end{figure}
\item
To obtain $\alpha$, we consider \autoref{eq:RecSymmetryThirdOrder} as overdetermined linear system (one equation for each value $\mu\in\{j\delta_x\mid j\in\{-512,\ldots,511\}\}$) for $\alpha^2(t)$ and $\frac{\alpha'(t)}{\alpha(t)}$, where the coefficients can be explicitly calculated with the values of $\varphi$ and $\omega_3$ obtained so far. Solving the corresponding normal equations, gives us the values $\alpha(t)$, see \autoref{fgRecAlpha}.
\begin{figure}
\begin{center}
\includegraphics{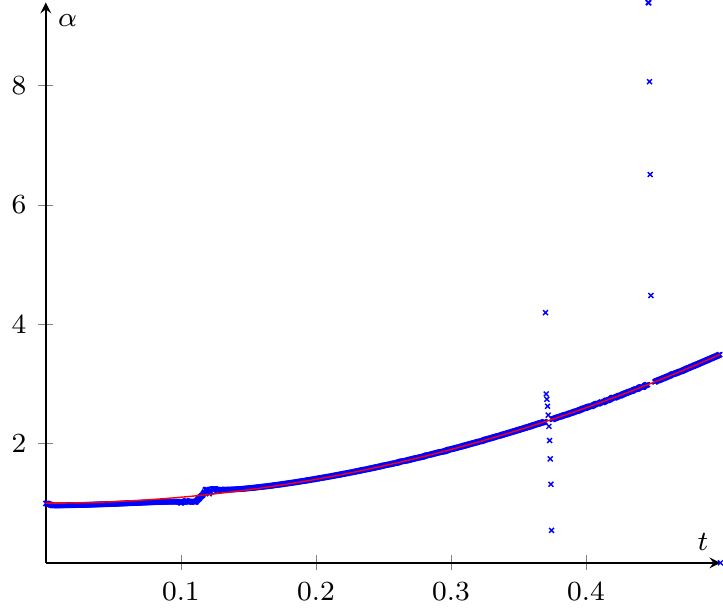}
\end{center}
\caption{Reconstruction of $\alpha$ (the blue crosses are the reconstructed values, the red curve is the exact function).}\label{fgRecAlpha}
\end{figure}
\end{enumerate}

Since we chose a sufficiently high resolution in time and space, the simulations did not require any sort of regularisation (we did not enforce continuity of the reconstructions in any way) and serve more as a test for the correctness of the formulas, see the plot of the reconstruction errors in \autoref{fgRecErrors}.
\begin{figure}
\begin{center}
\includegraphics{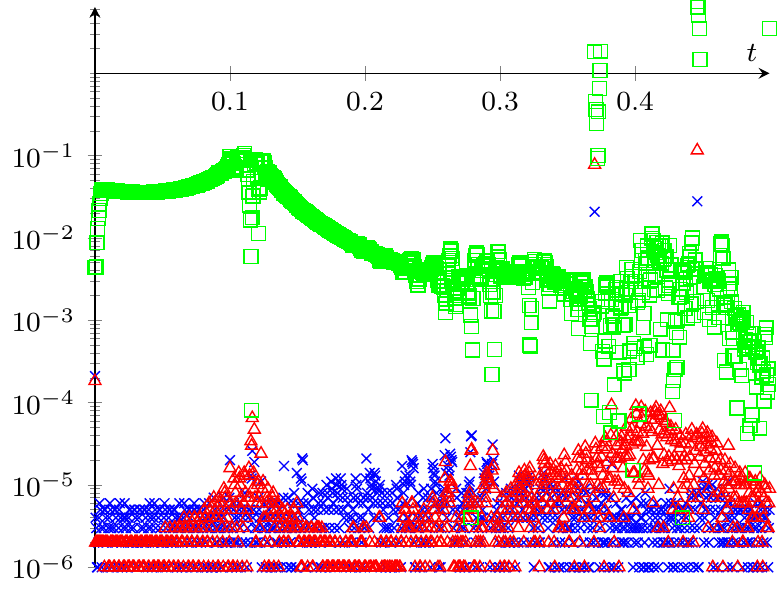}
\end{center}
\caption{Absolute errors in the reconstructions of $\varphi$ (the crosses), $\omega_3$ (the triangles), and $\alpha$ (the squares).}\label{fgRecErrors}
\end{figure}

\section*{Conclusions and Outlook}
Tomographic reconstruction of a trapped object has the intrinsic problem that the object's rotation is typically not 
as smooth and well-defined (and therefore known) as applied during tomographic data acquisition in other settings. 
Instead, the object undergoes a more irregular motion described as time-dependent translations and rotations. 

In the present work we deliver a first step into the direction of tomographic reconstruction of optically and/or 
acoustically trapped particles. We assume imaging at low numerical aperture of absorptive samples with amplitude contrast, which means 
that the attenuation projection describes the imaging process reasonably well. 

For this case, we demonstrate---by explicit reconstruction---how the motional parameters can be inferred, 
wherever permitted uniquely. However, not all parameters are uniquely retrievable, for example, the component of the translation in the direction the projection 
images are taken cannot be seen, or, in case of symmetries of the sample, we encounter fake solutions of our equations for the angular velocity (in the extreme case of a spherically symmetric object, no information on the motion is available). More uniqueness studies are on the way.

The mathematical problem formulation reveals similarities to single-particle Cryo-EM, however, here we can avoid the time-consuming 
step of labelling of attenuation projection images, because the labels can be identified with a time-stamp of the recordings over time.
Nevertheless, one could envision that our method could serve as a corrector method, when images in Cryo-EM are labelled by neighboring 
projection images, and difference quotients are used to approximate artificial time derivatives. 
In the future, the proposed motion estimation will be tested on video data acquired from biological samples held in 
optical and/or acoustic traps, see \autoref{fg:Pollen}. Moreover, it will be necessary to study corrections or alternative approaches required when going from 
attenuation projection images to optical images. Image formation of amplitude- or phase-samples in a microscope may significantly 
differ from attenuation projections, as explained in~\cite{Arr99,ArrScho09}.

\begin{figure}
\begin{center}
\includegraphics[width=0.75\textwidth]{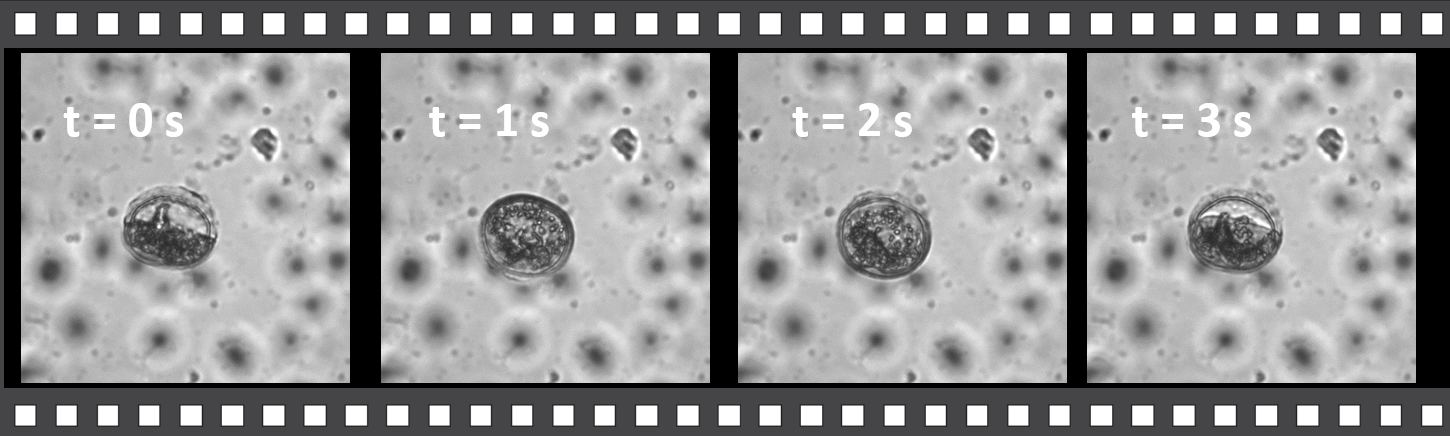}
\end{center}
\caption{Video frames from a video recorded of a pollen grain (Sansevieria trifasciata) rotating in an acoustic trap.}\label{fg:Pollen}
\end{figure}

\subsection*{Acknowledgements}
The authors are supported by the Austrian Science Fund (FWF), 
with SFB F68, project F6804-N36 (Coupled Physics Imaging), project F6806-N36 (Inverse Problems in Imaging of Trapped Particles), and project F6807-N36 (Tomography with Uncertainties). The authors thank Mia  Kv\r{a}le  L\o vmo and Benedikt Pressl for providing the video of the trapped pollen grain.

\appendix
\section{Appendix} \label{sec:app}
The main ingredient of our results for motion estimation in \autoref{section_method} are based on calculating 
higher order derivatives of composed functions. For this we require a tensor notation:

\subsection*{Tensor Derivatives}
Let 
\begin{equation*}
\begin{aligned}
 f: \R \times \R^2 &\to \C,\\
    (t,k) &\mapsto f(t,k)
\end{aligned}
\end{equation*}
and 
\begin{equation*}
\begin{aligned}
 g: \R &\to \R^2.\\
    t &\mapsto \begin{pmatrix} g_1(t)\\ g_2(t) \end{pmatrix}
\end{aligned}
\end{equation*}
In this paper we use derivatives of up to order $3$ of the composed function 
\begin{equation*}
 \begin{aligned}
 h: \R &\to \R.\\
    t &\mapsto h(t):=f(t,g(t))
\end{aligned}
\end{equation*}
These derivatives will be expressed via a tensor notation:
\begin{itemize}
 \item We denote by $\dk{i}{f}(t,\kappa) :\underbrace{\R^2 \times \R^2 \cdots \R^2}_{i \text{ times}} \to \C$ the derivative of order 
       $i=0,1,2,\ldots$ of the function $k \in \R^2 \mapsto f(t,k)$ with respect to the variable 
       $k$ for fixed $t$ at a point $\kappa \in \R^2$. We write the evaluation of the tensor $\dk{i}{f}(t,\kappa)$ in the form
       \begin{equation*}
        \dk{i}{f}(t,\kappa)\lsem v_1,v_2,\cdots,v_i \rsem.
       \end{equation*}
 \item The time derivatives 
       of the composed function $t \to h(t)$ up to order three then can be expressed with the simplifying notation 
       $\dk{i}{f} = \dk{i}{f}(t,g(t))$, as follows:
	\begin{equation} \label{eq:tot-abl} \begin{aligned}
	h'(t) = & \dk{1}{f} \lsem g' \rsem(t) +\partial_tf(t,g(t))\\
	h''(t) = &  \dk{2}{f} \lsem g',g'\rsem (t) +
	            \partial_t {\dk{1}{f} \lsem g'\rsem}(t) + 
	            \dk{1}{f} \lsem g'' \rsem(t) + \partial_{tt}f(t,g(t))\\
	h'''(t) = & \dk{3}{f} \lsem g',g',g'\rsem(t)  + 
	            \partial_t {\dk{2}{f}\lsem g',g'\rsem}(t) + 
	            \partial_{tt}{\dk{1}{f}\lsem g'\rsem}(t) +\\
	           & \quad 3 \biggl(\dk{2}{f}\lsem g',g''\rsem(t) +\partial_t {\dk{1}{f}\lsem g''\rsem (t)} \biggr)+
	             \dk{1}{f}\lsem g'''\rsem(t) +
	             \partial_{ttt}f(t,g(t)).
	\end{aligned}
	\end{equation}
	Note that for $i=1,2$ the tensor notation simplifies to the usual matrix vector multiplications:
	\begin{equation*}
	 \begin{aligned}
	     \dk{1}{f}\lsem g'\rsem(t) &= \nabla_k[f](t,g(t)) \cdot g'(t),\\  
	     \dk{2}{f}\lsem g',g'\rsem(t) &= g'(t)^T \; \nabla_k^2[f](t,g(t)) \; g'(t).\\           
	 \end{aligned}
	\end{equation*}       
\end{itemize}

\subsection*{Higher Order Expansion of Angular Velocities}
\begin{lemma}[Taylor Polynomials]
	\label{thTayPolynomial}
	Let $R\in\Rspace$, $t\in\R$ and $\omega\in\omspace$ be the angular velocity corresponding to $R$. 
	Then, we have
	\begin{align*}
	\frac1{t-s}P(e_3\times R(s)^{\mathrm T}R(t)e_3) &= \sum_{j=0}^3a_j(s)(t-s)^j+\tilde a(s,t), \\
	\frac1{s-t}P(e_3\times R(t)^{\mathrm T}R(s)e_3) &= \sum_{j=0}^3b_j(s)(t-s)^j+\tilde b(s,t)
	\end{align*}
	with remainder terms $a,b$ fulfilling $\lim_{t\to s}\frac{\tilde a(s,t)}{(t-s)^3}=\lim_{t\to s}\frac{\tilde b(s,t)}{(t-s)^3}=0$ and with the coefficients
	\begin{equation}\label{eqTayCoefficients}
	\begin{aligned}
	a_0 &= \alpha v, \\
	2a_1 &= \omega_3\alpha v^\perp+(\alpha v)', \\
	6a_2 &= -(\alpha^2+\omega_3^2)\alpha v+2\omega_3(\alpha v^\perp)'+\omega_3'\alpha v^\perp+(\alpha v)'', \\
	24a_3 &= -(\alpha^2+\omega_3^2)\omega_3\alpha v^\perp+2\omega_3\omega_3'\alpha v-2\omega_3^2(\alpha v)'-5(\alpha\alpha'+\omega_3\omega_3')\alpha v \\
	&\qquad\qquad+3\omega_3(\alpha v^\perp)''-(\alpha^2+\omega_3^2)(\alpha v)'+3\omega_3'(\alpha v^\perp)'+\omega_3''\alpha v^\perp+(\alpha v)''', \\
	b_0 &= \alpha v, \\
	2b_1 &= -\omega_3\alpha v^\perp+(\alpha v)', \\
	6b_2 &= -(\alpha^2+\omega_3^2)\alpha v-\omega_3(\alpha v^\perp)'-2\omega_3'\alpha v^\perp+(\alpha v)'', \\
	24b_3 &= (\alpha^2+\omega_3^2)\omega_3\alpha v^\perp-2\omega_3\omega_3'\alpha v+2\omega_3^2(\alpha v)'-3(\alpha\alpha'+\omega_3\omega_3')\alpha v \\
	&\qquad\qquad-\omega_3(\alpha v^\perp)''-3(\alpha^2+\omega_3^2)(\alpha v)'-3\omega_3'(\alpha v^\perp)'-3\omega_3''\alpha v^\perp+(\alpha v)'''.
	\end{aligned}
	\end{equation}
\end{lemma}

\begin{proof}
	We abbreviate the matrix-valued function $R^{\mathrm T}R'$ by $B$ and calculate from the identity $R'=RB$ further derivatives by differentiation:
	\begin{align*}
	R^{\mathrm T}R''&=B^2+B', \\
	R^{\mathrm T}R'''&=B^3+2BB'+B'B+B'', \\
	R^{\mathrm T}R''''&=B^4+3B^2B'+2BB'B+3BB''+B'B^2+3(B')^2+B''B+B'''.
	\end{align*}
	With this, we can calculate the coefficients
	\[ a_j=\frac1{(j+1)!}P\big(e_3\times R^{\mathrm T}R^{(j+1)}e_3\big),\qquad b_j=-\frac1{(j+1)!}P\big(e_3\times(R^{(j+1)})^{\mathrm T}Re_3\big), \]
	where we can use the antisymmetry $B=-B^{\mathrm T}$ for the calculatuon of the coefficients $b_j$.
	
	Plugging in \autoref{eqDefOmega} to write $Bx=\omega\times x$ and applying the identity $f\times(g\times h)=\left<f,h\right>g-\left<f,g\right>h$ for $f,g,h\in\R^3$, we find
	\begin{align*}
	a_0 &= P\big(\omega\big), \\
	2a_1 &= P\big(\omega_3e_3\times\omega+\omega'\big), \\
	6a_2 &= P\big(-|\omega|^2\omega+2\omega_3e_3\times\omega'+\omega_3'e_3\times\omega+\omega''\big), \\
	24a_3 &= P\big(-|\omega|^2\omega_3e_3\times\omega+3(\omega_3\omega_3'\omega-\omega_3^2\omega'-\left<\omega,\omega'\right>\omega)-2\left<\omega,\omega'\right>\omega+3\omega_3e_3\times\omega'' \\
	&\qquad\qquad+(\omega_3^2\omega'-\omega_3\omega_3'\omega-|\omega|^2\omega')+3\omega_3'e_3\times\omega'+\omega_3''e_3\times\omega+\omega'''\big) \\
	&= P\big(-|\omega|^2\omega_3e_3\times\omega+2\omega_3\omega_3'\omega-2\omega_3^2\omega'-5\left<\omega,\omega'\right>\omega+3\omega_3e_3\times\omega'' \\
	&\qquad\qquad-|\omega|^2\omega'+3\omega_3'e_3\times\omega'+\omega_3''e_3\times\omega+\omega'''\big), \\
	b_0 &= P\big(\omega\big), \\
	2b_1 &= P\big(-\omega_3e_3\times\omega+\omega'\big), \\
	6b_2 &= P\big(-|\omega|^2\omega-\omega_3e_3\times\omega'-2\omega_3'e_3\times\omega+\omega''\big), \\
	24b_3 &= P\big(|\omega|^2\omega_3e_3\times\omega+(\omega_3\omega_3'\omega-\omega_3^2\omega'-\left<\omega,\omega'\right>\omega)-2\left<\omega,\omega'\right>\omega-\omega_3e_3\times\omega'' \\
	&\qquad\qquad+3(\omega_3^2\omega'-\omega_3\omega_3'\omega-|\omega|^2\omega')-3\omega_3'e_3\times\omega'-3\omega_3''e_3\times\omega+\omega'''\big) \\
	&= P\big(|\omega|^2\omega_3e_3\times\omega-2\omega_3\omega_3'\omega+2\omega_3^2\omega'-3\left<\omega,\omega'\right>\omega-\omega_3e_3\times\omega'' \\
	&\qquad\qquad-3|\omega|^2\omega'-3\omega_3'e_3\times\omega'-3\omega_3''e_3\times\omega+\omega'''\big).
	\end{align*}
	Inserting $P(\omega)=\alpha v$ and using Taylor's theorem, we arrive at the stated Taylor polynomials.
\end{proof}

\begin{proof}[of \autoref{propD3}]
	We use again in \autoref{eqRecSymmetry} the Taylor polynomials
	\begin{align*}
	\frac1{t-s}P(e_3\times(R(s)^{\mathrm T}R(t)e_3)) &= a_0(s)+a_1(s)(t-s)+o(t-s)\quad\text{and} \\
	\frac1{s-t}P(e_3\times(R(t)^{\mathrm T}R(s)e_3)) &= b_0(s)+b_1(s)(t-s)+o(t-s)
	\end{align*}
	with the coefficients $a_j$ and $b_j$, $j=0,1$, calculated in \autoref{thTayPolynomial}. Then, by 
	taking three derivatives with respect to $t$  and evaluating at $t=s$ 
	(omitting the argument $s$ in the coefficients $a$ and $b$) we find with the notation of \autoref{eq:tot-abl} that
	\begin{align*}
	&\dk{3}{\tilde{\mathcal J}}(s,\lambda a_0)\lsem \lambda a_1,\lambda a_1,\lambda a_1 \rsem + 
	  6\dk{2}{\tilde{\mathcal J}}(s,\lambda a_0)\lsem \lambda a_1,\lambda a_2 \rsem + 
	  6\dk{1}{\tilde{\mathcal J}}(s,\lambda a_0) \lsem \lambda a_3 \rsem \\
	&\qquad = \dk{3}{\tilde{\mathcal J}}(s,\lambda b_0) \lsem \lambda b_1,\lambda b_1,\lambda b_1 \rsem + 
	6\dk{2}{\tilde{\mathcal J}}(s,\lambda b_0) \lsem \lambda b_1,\lambda b_2 \rsem + 
	6\dk{1}{\tilde{\mathcal J}}(s,\lambda b_0) \lsem \lambda b_3 \rsem \\
	&\qquad\qquad+3\partial_t\dk{2}{\tilde{\mathcal J}}(s,\lambda b_0) \lsem \lambda b_1,\lambda b_1 \rsem + 
	6\partial_t\dk{1}{\tilde{\mathcal J}}(s,\lambda b_0) \lsem \lambda b_2 \rsem \\
	&\qquad\qquad+3\partial_{tt} \dk{1}{\tilde{\mathcal J}}(s,\lambda b_0) \lsem \lambda b_1 \rsem +\partial_{ttt}\tilde{\mathcal J}(s,\lambda b_0).
	\end{align*}
	We calculate some of the terms separately:
	\begin{itemize}
		\item
		The third derivatives of $\tilde{\mathcal J}$ yield
		\begin{multline*}
		\dk{3}{\tilde{\mathcal J}}(s,\lambda a_0)\lsem \lambda a_1,\lambda a_1,\lambda a_1 \rsem
		-\dk{3}{\tilde{\mathcal J}}(s,\lambda b_0)\lsem \lambda b_1,\lambda b_1,\lambda b_1 \rsem \\
		= \frac14\lambda^3\left(\omega_3^3\dk{3}{\tilde{\mathcal J}}(s,\lambda\alpha v)\lsem \alpha v^\perp,\alpha v^\perp,\alpha v^\perp \rsem 
		+3\omega_3\dk{3}{\tilde{\mathcal J}}(s,\lambda\alpha v)\lsem \alpha v^\perp,(\alpha v)',(\alpha v)' \rsem\right).
		\end{multline*}
		\item
		For the second derivatives of $\tilde{\mathcal J}$, we find
		\begin{multline*}
		6\dk{2}{\tilde{\mathcal J}}(s,\lambda a_0)\lsem \lambda a_1,\lambda a_2 \rsem-6\dk{2}{\tilde{\mathcal J}}(s,\lambda b_0) \lsem \lambda b_1,\lambda b_2 \rsem \\
		=\lambda^2 \omega_3 \left(\dk{2}{\tilde{\mathcal J}}(s,\lambda\alpha v) \lsem \alpha v^\perp,(\alpha v)'' \rsem - 
		(\alpha^2+\omega_3^2)\dk{2}{\tilde{\mathcal J}}(s,\lambda\alpha v) \lsem \alpha v^\perp,\alpha v \rsem \right) \\
		+\frac12\lambda^2\omega_3\left(\omega_3\dk{2}{\tilde{\mathcal J}}(s,\lambda\alpha v) \lsem \alpha v^\perp,(\alpha v^\perp)'\rsem 
		- \omega_3'\dk{2}{\tilde{\mathcal J}}(s,\lambda\alpha v) \lsem \alpha v^\perp,\alpha v^\perp \rsem \right) \\
		+\frac32\lambda^2\left(\omega_3\dk{2}{\tilde{\mathcal J}}(s,\lambda\alpha v) \lsem (\alpha v)',(\alpha v^\perp)' \rsem 
		+\omega_3'\dk{2}{\tilde{\mathcal J}}(s,\lambda\alpha v) \lsem (\alpha v)',\alpha v^\perp \rsem \right).
		\end{multline*}
		\item
		For the first derivatives of $\tilde{\mathcal J}$, we get with 
		\begin{equation*} 
		  \begin{aligned}
		   \nu &:= -(\alpha^2+\omega_3^2)\omega_3\alpha v^\perp+2\omega_3\omega_3'\alpha v-2\omega_3^2(\alpha v)' \\
		       & \qquad -(\left<\alpha v,(\alpha v)'\right>+\omega_3\omega_3')\alpha v+2\omega_3(\alpha v^\perp)''\\
		       & \qquad + (\alpha^2+\omega_3^2)(\alpha v)'+3\omega_3'(\alpha v^\perp)'+2\omega_3''\alpha v^\perp,
		  \end{aligned}
                \end{equation*}
                that
                \begin{equation*}
		6\dk{1}{\tilde{\mathcal J}}(s,\lambda a_0)\lsem \lambda a_3 \rsem 
		-6\dk{1}{\tilde{\mathcal J}}(s,\lambda b_0)\lsem \lambda b_3 \rsem 
		=\frac12\lambda \dk{1}{\tilde{\mathcal J}}(s,\lambda\alpha v)
		  \lsem \nu \rsem.
		\end{equation*}
		\item
		On the other side, we find for the mixed derivatives of $\tilde{\mathcal J}$
		\begin{multline}\label{eqIdentity3}
		3\partial_t\dk{2}{\tilde{\mathcal J}}(s,\lambda b_0)\lsem\lambda b_1,\lambda b_1\rsem+6\partial_t\dk{1}{\tilde{\mathcal J}}(s,\lambda b_0)\lsem\lambda b_2\rsem+3\partial_{tt}\dk{1}{\tilde{\mathcal J}}(s,\lambda b_0)\lsem\lambda b_1\rsem+\partial_{ttt}\tilde{\mathcal J}(s,\lambda b_0) \\
		=\frac34\lambda^2\Big(\omega_3^2\partial_t\dk{2}{\tilde{\mathcal J}}(s,\lambda\alpha v)\lsem\alpha v^\perp,\alpha v^\perp\rsem\\
		-2\omega_3\partial_t\dk{2}{\tilde{\mathcal J}}(s,\lambda\alpha v)\lsem\alpha v^\perp,(\alpha v)'\rsem
		+\partial_t\dk{2}{\tilde{\mathcal J}}(s,\lambda\alpha v)\lsem(\alpha v)',(\alpha v)'\rsem\Big)\\
		+\lambda\partial_t\dk{1}{\tilde{\mathcal J}}(s,\lambda\alpha v)\lsem(\alpha v)''-2\omega_3'\alpha v^\perp-\omega_3(\alpha v^\perp)'-(\alpha^2+\omega_3^2)\alpha v\rsem \\
		+\frac32\lambda\partial_{tt}\dk{1}{\tilde{\mathcal J}}(s,\lambda\alpha v)\lsem(\alpha v)'-\omega_3\alpha v^\perp)+\partial_{ttt}\tilde{\mathcal J}(s,\lambda\alpha v\rsem.
		\end{multline}
	\end{itemize}
	
	The terms with the second order derivatives in $\alpha$ are given by
	\[ \lambda^2\omega_3\alpha\alpha''\dk{2}{\tilde{\mathcal J}}(s,\lambda\alpha v)\lsem v^\perp,v\rsem+\lambda\omega_3\alpha''\dk{1}{\tilde{\mathcal J}}(s,\lambda\alpha v)\lsem v^\perp\rsem-\lambda\alpha''\partial_t\dk{1}{\tilde{\mathcal J}}(s,\lambda\alpha v)\lsem v\rsem = 0 \]
	for $\alpha\ne0$, since, if we take the derivative of \autoref{eqRecSymmetryFirstOrder} with respect to $\lambda$, we find that
	\[ \alpha\partial_t\dk{1}{\tilde{\mathcal J}}(s,\lambda\alpha v)\lsem v\rsem=\lambda\omega_3\alpha^2\dk{2}{\tilde{\mathcal J}}(s,\lambda\alpha v)\lsem v^\perp,v\rsem+\omega_3\alpha \dk{1}{\tilde{\mathcal J}}(s,\lambda\alpha v)\lsem v^\perp\rsem. \]
	
	Taking the time derivative of \autoref{eqRecSymmetryFirstOrder}, we arrive at
	\begin{multline*}
	\partial_{tt}\tilde{\mathcal J}(s,\lambda\alpha v)+\lambda\partial_t\dk{1}{\tilde{\mathcal J}}(s,\lambda\alpha v)\lsem (\alpha v)'\rsem \\
	=\lambda\omega_3'\dk{1}{\tilde{\mathcal J}}(s,\lambda\alpha v)\lsem\alpha v^\perp\rsem+\lambda\omega_3\partial_t\dk{1}{\tilde{\mathcal J}}(s,\lambda\alpha v)\lsem\alpha v^\perp\rsem \\
	+\lambda\omega_3\dk{1}{\tilde{\mathcal J}}(s,\lambda\alpha v)\lsem(\alpha v^\perp)'\rsem+\lambda^2\omega_3\dk{2}{\tilde{\mathcal J}}(s,\lambda\alpha v)\lsem\alpha v^\perp,(\alpha v)'\rsem.
	\end{multline*}
	Another derivative with respect to time gives us
	\begin{multline}\label{eqV2}
	\partial_{ttt}\tilde{\mathcal J}(s,\lambda\alpha v)+2\lambda\partial_{tt}\dk{1}{\tilde{\mathcal J}}(s,\lambda\alpha v)\lsem(\alpha v)'\rsem+\lambda^2\partial_t\dk{2}{\tilde{\mathcal J}}(s,\lambda\alpha v)\lsem(\alpha v)',(\alpha v)'\rsem+\lambda\partial_t\dk{1}{\tilde{\mathcal J}}(s,\lambda\alpha v)\lsem(\alpha v)''\rsem \\
	=\lambda\omega_3''\dk{1}{\tilde{\mathcal J}}(s,\lambda\alpha v)\lsem\alpha v^\perp\rsem+2\lambda\omega_3'\partial_t\dk{1}{\tilde{\mathcal J}}(s,\lambda\alpha v)\lsem\alpha v^\perp\rsem+2\lambda^2\omega_3'\dk{2}{\tilde{\mathcal J}}(s,\lambda\alpha v)\lsem\alpha v^\perp,(\alpha v)'\rsem \\
	+2\lambda\omega_3'\dk{1}{\tilde{\mathcal J}}(s,\lambda\alpha v)\lsem(\alpha v^\perp)'\rsem+\lambda\omega_3\partial_{tt}\dk{1}{\tilde{\mathcal J}}(s,\lambda\alpha v)\lsem\alpha v^\perp\rsem+2\lambda^2\omega_3\partial_t\dk{2}{\tilde{\mathcal J}}(s,\lambda\alpha v)\lsem\alpha v^\perp,(\alpha v)'\rsem \\
	+2\lambda\omega_3\partial_t\dk{1}{\tilde{\mathcal J}}(s,\lambda\alpha v)\lsem(\alpha v^\perp)'\rsem+2\lambda^2\omega_3\dk{2}{\tilde{\mathcal J}}(s,\lambda\alpha v)\lsem(\alpha v^\perp)',(\alpha v)'\rsem+\lambda\omega_3\dk{1}{\tilde{\mathcal J}}(s,\lambda\alpha v)\lsem(\alpha v^\perp)''\rsem\\
	+\lambda^2\omega_3\dk{2}{\tilde{\mathcal J}}(s,\lambda\alpha v)\lsem\alpha v^\perp,(\alpha v)''\rsem+\lambda^3\omega_3\dk{3}{\tilde{\mathcal J}}(s,\lambda\alpha v)\lsem \alpha v^\perp,(\alpha v)',(\alpha v)' \rsem.
	\end{multline}
	
	Subtracting the identity (by our choice of the functions $v$ and $\omega_3$) of \autoref{eqV2} from \autoref{eqIdentity3}, we find, by putting everything together, the condition
	\begin{multline*}
	\frac14\lambda^3\left(\omega_3^3\dk{3}{\tilde{\mathcal J}}(s,\lambda\alpha v)\lsem \alpha v^\perp,\alpha v^\perp,\alpha v^\perp \rsem-\omega_3\dk{3}{\tilde{\mathcal J}}(s,\lambda\alpha v)\lsem \alpha v^\perp,(\alpha v)',(\alpha v)' \rsem\right) \\
	-\lambda^2\omega_3(\alpha^2+\omega_3^2)\dk{2}{\tilde{\mathcal J}}(s,\lambda\alpha v)\lsem\alpha v^\perp,\alpha v\rsem \\
	+\frac12\lambda^2\omega_3\left(\omega_3\dk{2}{\tilde{\mathcal J}}(s,\lambda\alpha v)\lsem\alpha v^\perp,(\alpha v^\perp)'\rsem-\omega_3'\dk{2}{\tilde{\mathcal J}}(s,\lambda\alpha v)\lsem\alpha v^\perp,\alpha v^\perp\rsem\right) \\
	-\frac12\lambda^2\left(\omega_3\dk{2}{\tilde{\mathcal J}}(s,\lambda\alpha v)\lsem(\alpha v)',(\alpha v^\perp)'\rsem+\omega_3'\dk{2}{\tilde{\mathcal J}}(s,\lambda\alpha v)\lsem(\alpha v)',\alpha v^\perp\rsem\right) \\
	+\frac12\lambda \dk{1}{\tilde{\mathcal J}}(s,\lambda\alpha v)\lsem-(\alpha^2+\omega_3^2)\omega_3\alpha v^\perp-(\left<\alpha v,(\alpha v)'\right>-\omega_3\omega_3')\alpha v+(\alpha^2-\omega_3^2)(\alpha v)'-\omega_3'(\alpha v^\perp)'\rsem \\
	=\lambda^2\Big(\frac34\omega_3^2\partial_t\dk{2}{\tilde{\mathcal J}}(s,\lambda\alpha v)\lsem\alpha v^\perp,\alpha v^\perp\rsem\\
	+\frac12\omega_3\partial_t\dk{2}{\tilde{\mathcal J}}(s,\lambda\alpha v)\lsem\alpha v^\perp,(\alpha v)'\rsem
	-\frac14\partial_t\dk{2}{\tilde{\mathcal J}}(s,\lambda\alpha v)\lsem(\alpha v)',(\alpha v)'\rsem\Big) \\
	+\lambda\partial_t\dk{1}{\tilde{\mathcal J}}(s,\lambda\alpha v)\lsem\omega_3(\alpha v^\perp)'-(\alpha^2+\omega_3^2)\alpha v\rsem
	-\frac12\lambda\partial_{tt}\dk{1}{\tilde{\mathcal J}}(s,\lambda\alpha v)\lsem(\alpha v)'+\omega_3\alpha v^\perp\rsem.
	\end{multline*}
	
	We sort the terms with respect to the derivative $\alpha'$ and write the equation in the form
	\begin{equation}\label{eqODE3}
	A_0(\lambda\alpha)+\alpha^2A_{02}(\lambda\alpha)+A_1(\lambda\alpha)\lambda\alpha'+A_{12}(\lambda\alpha)(\lambda\alpha')^2 = 0.
	\end{equation}
	\begin{itemize}
		\item
		The coefficient $A_{12}$ of $(\alpha')^2$ is given by
		\begin{equation}\label{eqA12}
		A_{12}(\mu) = -\frac14\mu\omega_3\dk{3}{\tilde{\mathcal J}}(s,\mu v)\lsem v^\perp,v,v\rsem-\frac12\omega_3\dk{2}{\tilde{\mathcal J}}(s,\mu v)\lsem v,v^\perp\rsem+\frac14\partial_t\dk{2}{\tilde{\mathcal J}}(s,\mu v)\lsem v,v\rsem.
		\end{equation}
		We take the derivative of \autoref{eqRecSymmetryFirstOrder} with respect to $\lambda$, giving us
		\begin{equation}\label{eqVlambda1}
		\alpha\partial_t\dk{1}{\tilde{\mathcal J}}(s,\lambda\alpha v)\lsem v\rsem=\lambda\omega_3\alpha^2\dk{2}{\tilde{\mathcal J}}(s,\lambda\alpha v)\lsem v^\perp,v\rsem+\omega_3\alpha \dk{1}{\tilde{\mathcal J}}(s,\lambda\alpha v)\lsem v^\perp\rsem,
		\end{equation}
		and taking another derivative with respect to $\lambda$, we find the identity
		\[ \alpha^2\partial_t\dk{2}{\tilde{\mathcal J}}(s,\lambda\alpha v)\lsem v,v\rsem = 2\omega_3\alpha^2\dk{2}{\tilde{\mathcal J}}(s,\lambda\alpha v)\lsem v^\perp,v\rsem+\lambda\omega_3\alpha^3\dk{3}{\tilde{\mathcal J}}(s,\lambda\alpha v)\lsem v^\perp,v,v\rsem. \]
		Comparing this with \autoref{eqA12}, we see that for $\alpha\ne0$:
		\[ A_{12} = 0. \]
		
		\item
		The coefficient $A_1$ of $\alpha'$ is
		\begin{multline*}
		A_1(\mu) = -\frac12\mu^2\omega_3\dk{3}{\tilde{\mathcal J}}(s,\mu v)\lsem v^\perp,v,v'\rsem+\frac12\mu\omega_3^2\dk{2}{\tilde{\mathcal J}}(s,\mu v)\lsem v^\perp,v^\perp\rsem \\
		-\frac12\mu\omega_3\dk{2}{\tilde{\mathcal J}}(s,\mu v)\lsem v',v^\perp\rsem-\frac12\mu\omega_3\dk{2}{\tilde{\mathcal J}}(s,\mu v)\lsem v,(v^\perp)'\rsem-\frac12\mu\omega_3'\dk{2}{\tilde{\mathcal J}}(s,\mu v)\lsem v,v^\perp\rsem \\
		-\frac12\dk{1}{\tilde{\mathcal J}}(s,\mu v)\lsem\omega_3^2v+\omega_3'v^\perp\rsem-\frac12\mu\omega_3\partial_t\dk{2}{\tilde{\mathcal J}}(s,\mu v)\lsem v^\perp,v\rsem+\frac12\mu\partial_t\dk{2}{\tilde{\mathcal J}}(s,\mu v)\lsem v,v'\rsem\\
		-\omega_3\partial_t\dk{1}{\tilde{\mathcal J}}(s,\mu v)\lsem v^\perp\rsem+\frac12\partial_{tt}\dk{1}{\tilde{\mathcal J}}(s,\mu v)\lsem v\rsem
		\end{multline*}
		
		If we compare this with the time derivative of the derivative of \autoref{eqRecSymmetryFirstOrder} with respect to $\mu$ with $\lambda=\frac\mu\alpha$, that is, the time derivative of
		\begin{equation}\label{eqVmu}
		\partial_t\dk{1}{\tilde{\mathcal J}}(s,\mu v)\lsem v\rsem=\mu\omega_3\dk{2}{\tilde{\mathcal J}}(s,\mu v)\lsem v^\perp,v\rsem+\omega_3\dk{1}{\tilde{\mathcal J}}(s,\mu v)\lsem v^\perp\rsem,
		\end{equation}
		which is
		\begin{multline*}
		\partial_{tt}\dk{1}{\tilde{\mathcal J}}(s,\mu v)\lsem v\rsem+\mu\partial_t\dk{2}{\tilde{\mathcal J}}(s,\mu v)\lsem v,v'\rsem+\partial_t\dk{1}{\tilde{\mathcal J}}(s,\mu v)\lsem v'\rsem \\
		= \mu\omega_3'\dk{2}{\tilde{\mathcal J}}(s,\mu v)\lsem v^\perp,v\rsem+\mu^2\omega_3\dk{3}{\tilde{\mathcal J}}(s,\mu v)\lsem v^\perp,v,v'\rsem+\mu\omega_3\dk{2}{\tilde{\mathcal J}}(s,\mu v)\lsem(v^\perp)',v\rsem \\
		+\omega_3'\dk{1}{\tilde{\mathcal J}}(s,\mu v)\lsem v^\perp\rsem+2\mu\omega_3\dk{2}{\tilde{\mathcal J}}(s,\mu v)\lsem v^\perp,v'\rsem+\omega_3\dk{1}{\tilde{\mathcal J}}(s,\mu v)\lsem(v^\perp)'\rsem \\
		+\mu\omega_3\partial_t\dk{2}{\tilde{\mathcal J}}(s,\mu v)\lsem v^\perp,v\rsem+\omega_3\partial_t\dk{1}{\tilde{\mathcal J}}(s,\mu v)\lsem v^\perp\rsem;
		\end{multline*}
		then we see that some of the coefficients cancel each other and we are left with
		\begin{multline*}
		A_1(\mu) = \frac12\mu\omega_3^2\dk{2}{\tilde{\mathcal J}}(s,\mu v)\lsem v^\perp,v^\perp\rsem+\frac12\mu\omega_3\dk{2}{\tilde{\mathcal J}}(s,\mu v)\lsem v',v^\perp\rsem-\frac12\omega_3^2\dk{1}{\tilde{\mathcal J}}(s,\mu v)\lsem v\rsem \\
		-\frac12\omega_3\partial_t\dk{1}{\tilde{\mathcal J}}(s,\mu v)\lsem v^\perp\rsem-\frac12\partial_t\dk{1}{\tilde{\mathcal J}}(s,\mu v)\lsem v'\rsem+\frac12\omega_3\dk{1}{\tilde{\mathcal J}}(s,\mu v)\lsem(v^\perp)'\rsem.
		\end{multline*}
		We rewrite it in the form
		\begin{multline*}
		A_1(\mu) = \frac12\Big[\mu\omega_3\dk{2}{\tilde{\mathcal J}}(s,\mu v)\lsem v^\perp,\omega_3v^\perp+v'\rsem \\
		+\omega_3\dk{1}{\tilde{\mathcal J}}(s,\mu v)\lsem (v^\perp)'-\omega_3v\rsem-\partial_t\dk{1}{\tilde{\mathcal J}}(s,\mu v)\lsem\omega_3v^\perp+v'\rsem\Big].
		\end{multline*} 
		Remarking that $|v|^2=1$ and thus $\left<v,v'\right>=0$, we have a scalar function $\sigma$ such that
		\[ v' = \sigma v^\perp. \]
		With this, $A_1$ becomes (using $(v^\perp)' = (v')^\perp = -\sigma v$)
		\[ A_1(\mu) = \frac12(\omega_3+\sigma)\Big[\mu\omega_3\dk{2}{\tilde{\mathcal J}}(s,\mu v)\lsem v^\perp,v^\perp\rsem-\omega_3\dk{1}{\tilde{\mathcal J}}(s,\mu v)\lsem v \rsem-\partial_t\dk{1}{\tilde{\mathcal J}}(s,\mu v)\lsem v^\perp\rsem\Big]. \]	
		Thus, $A_1(\mu)$ is zero if either $\sigma=-\omega_3$, that is $v'=-\omega_3v^\perp$, in which case the bracket can be written as
		\[ -\frac{\d}{\d s} \dk{1}{\tilde{\mathcal J}}(s,\mu v)\lsem v^\perp\rsem = -\partial_t\dk{1}{\tilde{\mathcal J}}(s,\mu v)\lsem v^\perp\rsem-\mu\sigma \dk{2}{\tilde{\mathcal J}}(s,\mu v)\lsem v^\perp,v^\perp\rsem+\sigma \dk{1}{\tilde{\mathcal J}}(s,\mu v)\lsem v\rsem; \]
		or if the bracket vanishes, that is,
		\[ (\omega_3-\sigma)\partial_t\dk{1}{\tilde{\mathcal J}}(s,\mu v)\lsem v^\perp\rsem+\omega_3\frac{\d}{\d s} \dk{1}{\tilde{\mathcal J}}(s,\mu v)\lsem v^\perp\rsem = 0. \]
		
		Taking the time derivative of \autoref{eqRecSymmetryFirstOrder} with $\lambda=\frac\mu\alpha$, we get
		\begin{multline*}
		\partial_{tt}\tilde{\mathcal J}(s,\mu v)+\mu\partial_t\dk{1}{\tilde{\mathcal J}}(s,\mu v)\lsem v' \rsem = 
		\mu\omega_3'\dk{1}{\tilde{\mathcal J}}(s,\mu v)\lsem v^\perp\rsem +\mu\omega_3\partial_t\dk{1}{\tilde{\mathcal J}}(s,\mu v)\lsem v^\perp\rsem\\
		+\mu^2\omega_3\dk{2}{\tilde{\mathcal J}}(s,\mu v)\lsem v^\perp,v'\rsem+\mu\omega_3\dk{1}{\tilde{\mathcal J}}(s,\mu v)\lsem(v^\perp)'\rsem.
		\end{multline*}
		
		\item
		The coefficient $A_{02}$ of $\alpha^2$ consists of the terms
		\begin{multline*}
		A_{02}(\mu) = -\mu^2\omega_3\dk{2}{\tilde{\mathcal J}}(s,\mu v)\lsem v^\perp,v\rsem-\frac12\mu\omega_3\dk{1}{\tilde{\mathcal J}}(s,\mu v)\lsem v^\perp\rsem\\
		+\frac12\mu \dk{1}{\tilde{\mathcal J}}(s,\mu v)\lsem v'\rsem+\mu\partial_t\dk{1}{\tilde{\mathcal J}}(s,\mu v)\lsem v\rsem.
		\end{multline*} 
		With \autoref{eqVlambda1} this becomes
		\[ A_{02}(\mu) = \frac12\mu \dk{1}{\tilde{\mathcal J}}(s,\mu v)\lsem \omega_3v^\perp+v'\rsem = \frac12\mu(\omega_3+\sigma)\dk{1}{\tilde{\mathcal J}}(s,\mu v)\lsem v^\perp\rsem, \]
		which again vanishes for $\sigma=-\omega_3$.
		
		\item
		Finally, the coefficient $A_0$ is given by
		\begin{multline*}
		A_0(\mu) = \frac14\mu^3\omega_3^3\dk{3}{\tilde{\mathcal J}}(s,\mu v)\lsem v^\perp,v^\perp,v^\perp\rsem-\frac14\mu^3\omega_3\dk{3}{\tilde{\mathcal J}}(s,\mu v)\lsem v^\perp,v',v'\rsem\\
		-\mu^2\omega_3^3\dk{2}{\tilde{\mathcal J}}(s,\mu v)\lsem v^\perp,v\rsem 
		+\frac12\mu^2\omega_3\dk{2}{\tilde{\mathcal J}}(s,\mu v)\lsem\omega_3v^\perp-v',(v^\perp)'\rsem\\
		-\frac12\mu^2\omega_3'\dk{2}{\tilde{\mathcal J}}(s,\mu v)\lsem \omega_3v^\perp+v',v^\perp\rsem 
		-\frac12\mu\omega_3^2\dk{1}{\tilde{\mathcal J}}(s,\mu v)\lsem\omega_3v^\perp+v'\rsem\\
		+\frac12\mu\omega_3'\dk{1}{\tilde{\mathcal J}}(s,\mu v)\lsem\omega_3v-(v^\perp)'\rsem 
		-\frac34\mu^2\omega_3^2\partial_t\dk{2}{\tilde{\mathcal J}}(s,\mu v)\lsem v^\perp,v^\perp\rsem\\
		+\frac14\mu^2\partial_t\dk{2}{\tilde{\mathcal J}}(s,\mu v)\lsem v'-2\omega_3v^\perp,v'\rsem 
		-\mu\omega_3\partial_t\dk{1}{\tilde{\mathcal J}}(s,\mu v)\lsem (v^\perp)'-\omega_3v\rsem\\
		+\frac12\mu\partial_{tt}\dk{1}{\tilde{\mathcal J}}(s,\mu v)\lsem v'+\omega_3 v^\perp\rsem.
		\end{multline*}
		Writing again $v'=\sigma v^\perp$, this becomes
		\begin{multline*}
		A_0(\mu) = \frac14\mu^3\omega_3(\omega_3^2-\sigma^2)\dk{3}{\tilde{\mathcal J}}(s,\mu v)\lsem v^\perp,v^\perp,v^\perp\rsem \\
		-\frac12\mu^2\omega_3(2\omega_3^2+\omega_3\sigma-\sigma^2)\dk{2}{\tilde{\mathcal J}}(s,\mu v)\lsem v^\perp,v\rsem-\frac12\mu^2(\omega_3+\sigma)\omega_3'\dk{2}{\tilde{\mathcal J}}(s,\mu v)\lsem v^\perp,v^\perp\rsem\\
		-\frac12\mu\omega_3^2(\omega_3+\sigma)\dk{1}{\tilde{\mathcal J}}(s,\mu v)\lsem v^\perp\rsem+\frac12\mu(\omega_3+\sigma)\omega_3'\dk{1}{\tilde{\mathcal J}}(s,\mu v)\lsem v\rsem \\
		-\frac14\mu^2(3\omega_3^2+2\omega_3\sigma-\sigma^2)\partial_t\dk{2}{\tilde{\mathcal J}}(s,\mu v)\lsem v^\perp,v^\perp\rsem \\
		+\mu\omega_3(\omega_3+\sigma)\partial_t\dk{1}{\tilde{\mathcal J}}(s,\mu v)\lsem v\rsem+\frac12\mu(\omega_3+\sigma)\partial_{tt}\dk{1}{\tilde{\mathcal J}}(s,\mu v)\lsem v^\perp\rsem.
		\end{multline*}
		Subtracting \autoref{eqVmu}, we get rid of one term:
		\begin{multline*}
		A_0(\mu) = \frac14\mu(\omega_3+\sigma)\Big[\mu^2\omega_3(\omega_3-\sigma)\dk{3}{\tilde{\mathcal J}}(s,\mu v)\lsem v^\perp,v^\perp,v^\perp\rsem \\
		+2\mu \dk{2}{\tilde{\mathcal J}}(s,\mu v)\lsem v^\perp,\omega_3\sigma v-\omega_3'v^\perp\rsem+2\dk{1}{\tilde{\mathcal J}}(s,\mu v)\lsem \omega_3^2v^\perp+\omega_3'v\rsem \\
		-\mu(3\omega_3-\sigma)\partial_t\dk{2}{\tilde{\mathcal J}}(s,\mu v)\lsem v^\perp,v^\perp\rsem+2\partial_{tt}\dk{1}{\tilde{\mathcal J}}(s,\mu v)\lsem v^\perp\rsem\Big].
		\end{multline*}
	\end{itemize}
\end{proof}

\section*{References}
\renewcommand{\i}{\ii}
\printbibliography[heading=none]

\end{document}